\documentclass[11pt]{article}
\usepackage{amsmath,amssymb,amsthm,mathrsfs}
\usepackage[all]{xy}
\usepackage[dvips]{graphicx}

\newtheorem{Def}{Definition}[section]
\newtheorem{Thm}[Def]{Theorem}
\newtheorem{Prop}[Def]{Proposition}
\newtheorem{Cor}[Def]{Corollary}
\newtheorem{Lem}[Def]{Lemma}
\newtheorem{Rem}[Def]{Remark}
\newtheorem{Ex}[Def]{Example}

\newcommand{\C}{\mathbb{C}}
\newcommand{\R}{\mathbb{R}}
\newcommand{\Z}{\mathbb{Z}}
\newcommand{\N}{\mathbb{N}}
\newcommand{\Q}{\mathbb{Q}}
\newcommand{\PP}{\mathbb{P}}
\newcommand{\HH}{\mathbb{H}}

\newcommand{\GL}{\mathop{\mathrm{GL}}\nolimits}
\newcommand{\PGL}{\mathop{\mathrm{PGL}}\nolimits}

\newcommand{\id}{\mathop{\mathrm{id}}\nolimits}
\newcommand{\disc}{\mathop{\mathrm{disc}}\nolimits}
\newcommand{\Ker}{\mathop{\mathrm{Ker}}\nolimits}

\newcommand{\rk}{\mathop{\mathrm{rank}}\nolimits}

\newcommand{\Hom}{\mathop{\mathrm{Hom}}\nolimits}
\newcommand{\End}{\mathop{\mathrm{End}}\nolimits}
\newcommand{\Aut}{\mathop{\mathrm{Aut}}\nolimits}

\newcommand{\ord}{\operatorname{ord}}
\newcommand{\rank}{\operatorname{rank}}

\begin{document}
\title{Calabi--Yau Threefolds of Type K (II): Mirror Symmetry}
\author{Kenji Hashimoto \ \ \ \ Atsushi Kanazawa}
\date{}

\maketitle

\begin{abstract}
A Calabi--Yau threefold is called of type K if it admits an \'etale Galois covering by the product of a K3 surface and an elliptic curve. 
In our previous paper \cite{HK}, based on Oguiso--Sakurai's fundamental work \cite{OS}, 
we have provided the full classification of Calabi--Yau threefolds of type K and have studied some basic properties thereof. 
In the present paper, we continue the study, investigating them from the viewpoint of mirror symmetry. 
It is shown that mirror symmetry relies on duality of certain sublattices in the second cohomology of the K3 surface appearing in the minimal splitting covering.   
The duality may be thought of as a version of the lattice duality of the anti-symplectic involution on K3 surfaces discovered by Nikulin \cite{Ni3}. 
Based on the duality, we obtain several results parallel to what is known for Borcea--Voisin threefolds. 
Along the way, we also investigate the Brauer groups of Calabi--Yau threefolds of type K.  
\end{abstract}

%%%%%%%%%%%%%%%%%%%%%%%%%%%%%%%%%%%%%%%%%%%%%%%%%%%%%%%%%%%%%%%%%%%%%%%%%%%%%%%%%%%%%%%%%%%%%%%%%%%%%%%%%%%%%
%%%%%%%%%%%%%%%%%%%%%%%%%%%%%%%%%%%%%%%%%%%%%%%%%%%%%%%%%%%%%%%%%%%%%%%%%%%%%%%%%%%%%%%%%%%%%%%%%%%%%%%%%%%%%

%\subjclass[2010]{14J32, 14F33} 
%\keywords{Calabi--Yau threefold, K3 surface, fundamental group, lattice theory, Yukawa coupling, SYZ conjecture, special Lagrangian fibration}
%\tableofcontents

%%%%%%%%%%%%%%%%%%%%%%%%%%%%%%%%%%%%%%%%%%%%%%%%%%%%%%%%%%%%%%%%%%%%%%%%%%%%%%%%%%%%%%%%%%%%%%%%%%%%%%%%%%%%%
%%%%%%%%%%%%%%%%%%%%%%%%%%%%%%%%%%%%%%%%%%%%%%%%%%%%%%%%%%%%%%%%%%%%%%%%%%%%%%%%%%%%%%%%%%%%%%%%%%%%%%%%%%%%%

\section{Introduction}
The present paper studies mirror symmetry and some topological properties of Calabi--Yau threefolds of type K. 
A Calabi--Yau threefold\footnote{
We do not assume that $X$ is simply-connected. 
One reason is that simply-connected Calabi--Yau threefolds are not closed under mirror symmetry; 
a mirror partner of a simply-connected Calabi--Yau threefold may not be simply-connected \cite{GP}} 
%In view of mirror symmetry, our definition seems more natural than {\it strict} Calabi--Yau threefolds, which are simply-connected.} 
is a compact K\"ahler threefold $X$ with trivial canonical bundle $\varOmega_{X}^{3}\cong \mathcal{O}_{X}$ and $H^{1}(X, \mathcal{O}_{X})=0$. 
In \cite{OS}, Oguiso and Sakurai call $X$ a Calabi--Yau threefold  {\it of type K} if it admits an \'etale Galois covering by the product of a K3 surface and an elliptic curve. 
Among many candidates of such coverings, we can always find a unique smallest one, up to isomorphism as a covering, and we call it the minimal splitting covering. 
The importance of this class of Calabi--Yau threefolds comes from the fact that it is one of the two classes of Calabi--Yau threefolds with infinite fundamental group.  
In our previous work \cite{HK}, based on Oguiso--Sakurai's fundamental work, we have given the full classification of Calabi--Yau threefolds of type K.  

\begin{Thm}[\cite{OS,HK}] 
There exist exactly eight Calabi--Yau threefolds of type K, up to deformation equivalence.
The equivalence class is uniquely determined by the Galois group of the minimal splitting covering, 
which is isomorphic to one of the following combinations of cyclic and dihedral groups:
$$
C_{2},\ C_2 \times C_2, \ C_2 \times C_2 \times C_2, \ D_{6}, \ D_{8}, \ D_{10}, \ D_{12}, \ \text{and} \ \ C_2 \times D_8. 
$$ 
The Hodge numbers $h^{1,1}=h^{2,1}$ are respectively given by $11,7,5,5,4,3,3,3$.  
\end{Thm}

In \cite{HK}, we also obtain explicit presentations of the eight Calabi--Yau threefolds of type K. 
Although Calabi--Yau threefolds of type K are very special, 
their explicit nature makes them an exceptionally good laboratory for general theories and conjectures. 
Indeed, the simplest example, known as the Enriques Calabi--Yau threefold, 
has been one of the most tractable compact Calabi--Yau threefolds (for example \cite{FHSV,KM}).   
The objective of this paper is to investigate Calabi--Yau threefolds of type K with a view toward self-mirror symmetry. 

Here is a brief summary of the paper. 
Let $X$ be a Calabi--Yau threefold of type K and $\pi \colon S\times E \rightarrow X$ its minimal splitting covering, with a K3 surface $S$ and an elliptic curve $E$. 
The geometry of $X$ is equivalent to the $G$-equivariant geometry of the covering space $S\times E$, where $G:=\mathrm{Gal}(\pi)$ is the Galois group of the covering $\pi$.  
Moreover, $G$ turns out to be of the form $G=H\rtimes C_2$ and it acts on each factor, $S$ and $E$. 
Let $M_G:=H^2(S,\Z)^G$ and $N_G:=H^2(S,\Z)^H_{C_2}$ (see Section \ref{SECT_lattice} for the notation).  
The former represents the $G$-equivariant algebraic cycles and the latter the transcendental cycles. 
The first main result is an existence of duality between these lattices. 
\begin{Thm}[Theorem \ref{G-inv lattice}] \label{intro: main1}
There exists a lattice isomorphism $U\oplus M_G \cong N_G$ over the rational numbers $\Q$ (or some extension of $\Z$), where $U$ is the hyperbolic lattice. 
\end{Thm}
Recall that Calabi--Yau threefolds of type K are close cousins of Borcea--Voisin threefolds \cite{Bo,Vo}, 
whose mirror symmetry stems from the lattice duality of the anti-symplectic involution on K3 surfaces discovered by Nikulin \cite{Ni3}. 
Theorem \ref{intro: main1} can be thought of as an $H$-equivariant version of Nikulin's duality, although it does not hold over $\Z$. 
It also indicates the fact that $X$ is self-mirror symmetric\footnote{It is important to note that the K3 surface $S$ is in not self-mirror symmetry in the sense of Dolgachev \cite{Do} 
unless $G=C_2$, i.e.\ $H$ is trivial.}.  
Based on this fundamental duality, we obtain several results on the Yukawa couplings (Theorem \ref{Yukawa}) and special Lagrangian fibrations (Propositions \ref{SLAG} \& \ref{SLAG2})
parallel to what is known for Borcea--Voisin threefolds (Voisin \cite{Vo} and Gross--Wilson \cite{GW}).   
The mirror symmetry for Borcea--Voisin threefolds has a different flavour from that for the complete intersection Calabi--Yau threefolds in toric varieties and homogeneous spaces,  
and hence it has been a very important source of examples beyond the Batyrev--Borisov toric mirror symmetry. 
We hope that our work provides new examples of interesting mirror symmetry. 

We also investigate Brauer groups, which are believed to play an important role in mirror symmetry but have not been much explored in the literature (for example \cite{AM,BK}). 
An importance of the Brauer group $\mathrm{Br}(X)$ of a Calabi--Yau threefold $X$ lies in the fact \cite{Ad} that 
it is intimately related to another torsion group $H_1(X,\Z)$ and the derived category $\mathrm{D^bCoh}(X)$. 
An explicit computation shows the second main result. 
\begin{Thm}[Theorem \ref{Thm: brauer_grp}]
Let $X$ be a Calabi--Yau threefold of type K, then
$\mathrm{Br}(X)\cong \Z_2^{\oplus m}$, where $m$ is given by the
following.
%Here $H_1(X,\Z)$ is obtained in our previous work \cite{HK}.
\begin{equation*}
\begin{array}{|c|c|c|c|c|c|c|c|c|}
 \hline
 G & C_{2} & C_2 \times C_2 & C_2 \times C_2 \times C_2 & D_{6} &
D_{8} & D_{10} & D_{12} & C_2 \times D_8 \\ \hline
 m & 1 & 2 & 3 & 1 & 2 & 1 & 2 & 3 \\ \hline
\end{array}
\end{equation*}
\end{Thm}

As an application, we show that any derived equivalent Calabi--Yau threefolds of type K have isomorphic Galois groups of the minimal splitting coverings (Corollary \ref{cor: derived equiv}).
It is also interesting to observe that $H_1(X,\Z) \cong \mathrm{Br}(X) \oplus \Z^{\oplus2}_2$ holds in our case.
The role of the Brauer group in mirror symmetry is tantalising, and deserves further explorations.
For example, Gross discusses in \cite{Gr2} the subject in the context of SYZ mirror symmetry \cite{SYZ}.
Unfortunately, we are not able to unveil the role of Brauer group in mirror symmetry of a Calabi--Yau threefolds of type K at this point. 
We hope that our explicit computation is useful for future investigation. 
%We hope complete examples of special Lagrangian fibrations provided in this paper will be useful for developing general theory (for example \cite{Gr1,Gr2}).
%Especially, the role of the Brauer group in mirror symmetry is tantalising, and deserves further explorations.   

\subsection*{Structure of Paper}  
Section 2 sets conventions and recalls some basics of lattices and K3 surfaces.  
Section 3 begins with a brief review of the classification \cite{HK} of Calabi--Yau threefolds of type K and then 
investigates their topological properties, in particular their Brauer groups.  
Section 4 is devoted to the study of mirror symmetry, probing the A- and B-Yukawa couplings.  
Section 5 explores special Lagrangian fibrations inspired by the Strominger--Yau--Zaslow conjecture \cite{SYZ}.

\subsection*{Acknowledgement}
The authors are grateful to S. Hosono, J. Keum, and S. -T. Yau for inspiring discussions. 
Special thanks go to M. Gross for pointing out the importance of the Brauer groups. 
Much of this research was done when the first author was supported by Korea Institute for Advanced Study. 
The second author is supported by the Center of Mathematical Sciences and Applications at Harvard University. 

%%%%%%%%%%%%%%%%%%%%%%%%%%%%%%%%%%%%%%%%%%%%%%%%%%%%%%%%%%%%%%%%%%%%%%%%%%%%%%%%%%%%%%%%%%%%%%%%%%%%%%%%%%%%%
%%%%%%%%%%%%%%%%%%%%%%%%%%%%%%%%%%%%%%%%%%%%%%%%%%%%%%%%%%%%%%%%%%%%%%%%%%%%%%%%%%%%%%%%%%%%%%%%%%%%%%%%%%%%%

\section{Lattices and K3 surfaces}
In this section we summarize some basics of lattices and K3 surfaces, following \cite{BHPV, Ni2}. 
%This will also serve to set notations. 
%As general references the reader might consult references \cite{BHPV, Ni2}.

\subsection{Lattices} \label{SECT_lattice}

A lattice is a free $\Z$-module $L$ of finite rank together with a symmetric bilinear form $\langle *,**\rangle\colon L\times L\to \Z$. 
By an abuse of notation, we often denote a lattice simply by $L$. 
Given a basis, the bilinear form is represented by a Gram matrix and the discriminant $\disc(L)$ is the determinant of the Gram matrix.  
We define $L(\lambda)$ to be the lattice obtained by multiplying the bilinear form by $\lambda \in \Q$. 
We denote by $\langle a \rangle$ the lattice of rank $1$ generated by $x$ with $x^2:=\langle x, x\rangle=a$. 
A lattice $L$ is called even if $x^2 \in 2\mathbb{Z}$ for all $x \in L$.  
%We say that $x \in L$ is $n$-divisible, denoted $\mathrm{div}(x)=n$, if $\langle x,L\rangle = n\Z$. 
$L$ is non-degenerate if $\disc(L)\neq 0$ and unimodular if $\disc(L)=\pm1$.
If $L$ is a non-degenerate lattice, the signature of $L$ is the pair $(t_{+},t_{-})$ where $t_+$ and $t_-$ respectively denote
 the numbers of positive and negative eigenvalues of the Gram matrix.
% the dimensions of 
%the positive and negative eigenspaces of $L\otimes \mathbb{R}$.
%We define $\sign L:=t_+-t_-$.

A sublattice $M$ of a lattice $L$ is a submodule of $L$ with the bilinear form of $L$ restricted to $M$.
It is called primitive if $L/M$ is torsion free. 
% maybe we should define orthogonal sublattice and primitive closure? 
We denote by $M^{\bot}_L$ (or simply $M^\bot$) the orthogonal complement of $M$ in $L$. 
%$$M^{\bot}_L:= \{x\in L \bigm| \langle x,y \rangle=0 \  (\forall y\in M) \}. $$
We always assume that an action of a group $G$ on a lattice $L$ preserves the bilinear form. 
Then the invariant part $L^G$ and the coinvariant part $L_G$ are defined as
$$
L^G:=\{x\in L \bigm| g\cdot x=x \  (\forall g\in G) \},\quad
L_G:=(L^G)^\bot_L.
$$
We simply denote $L^{\langle g\rangle}$ and $L_{\langle g\rangle}$ by $L^{g}$ and $L_{g}$ respectively for $g \in G$.
If another group $H$ acts on $L$, we denote $L^G \cap L_H$ by $L^G_H$.

The hyperbolic lattice $U$ is the lattice given by the Gram matrix $\begin{bmatrix} 0 & 1\\ 1 & 0\\ \end{bmatrix}$.
The corresponding basis $e,f$ is called the standard basis. 
%$U$ is an indefinite even unimodular lattice.
Let $A_{m}, \ D_{n}, \ E_{l}, \ (m\ge 1, \ n\ge 4, \ l =6,7,8)$ be the lattices defined by the corresponding Cartan matrices. 
%Among others, the lattice $E_{8}$ plays an important role in the study of K3 surfaces. 
%$E_{8}$ is the unique positive definite, even unimodular lattice of rank $8$ up to isomorphism. 
%By the uniqueness of indefinite even unimodular lattices,
Every indefinite even unimodular lattice is realized as an orthogonal sum of copies of $U$ and $E_{8}(\pm 1)$ in an essentially unique way, 
the only relation being $E_{8}\oplus E_{8}(-1)\cong U^{\oplus 8}$.  
Hence an even unimodular lattice of signature $(3,19)$ is isomorphic to $U^{\oplus 3} \oplus E_{8} (-1)^{\oplus 2}$, which is called the K3 lattice. 

%For indefinite unimodular lattices, the classification is easy to describe. 
%For $n,m \ge 1$, there exists a unique odd unimodular lattice $I_{n,m}:=\langle 1 \rangle ^{\oplus n}\oplus \langle -1 \rangle ^{\oplus m}$ of signature $ (n,m)$,  up to isomorphism. 
%The classification of indefinite even unimodular lattices is also very simple. 
%Every indefinite even unimodular lattice can be realized as an orthogonal sum of copies of $U$ and $E_{8}(\pm 1)$ in an essentially unique way, 
%the only relation being $E_{8}\oplus E_{8}(-1)\cong U^{\oplus 8}$. \\

Let $L$ be a non-degenerate even lattice.
We have a natural identification
\begin{equation*}
 L^{\vee}:=\mathrm{Hom}(L,\mathbb{Z})=
 \{ x \in L \otimes \Q \bigm| \langle x,y \rangle\in\Z \  (\forall y\in L) \}.
\end{equation*}
%The bilinear form of $L$ determines a canonical embedding $L \hookrightarrow L^{\vee}:=\mathrm{Hom}(L,\mathbb{Z})$. 
The discriminant group $A(L):=L^{\vee}/L$ is a finite abelian group of order $|\disc(L)|$ 
equipped with a quadratic map $q(L)\colon A(L) \rightarrow \mathbb{Q}/2\mathbb{Z}$ given by $x+L \mapsto x^2 + 2\mathbb{Z}$. 
%We denote by $l(L)$  the number of minimal generators of $A_{L}$. 
%The bilinear form on $L$ linearly extends to the one on the dual $L^{\vee}$.
%We define a quadratic map $q(L)\colon A(L) \rightarrow \mathbb{Q}/2\mathbb{Z}$ by $x+L \mapsto x^2 + 2\mathbb{Z}$. 
%Two even lattices $L$ and $L'$ have isomorphic discriminant form 
%if and only if they are stably equivalent, that is, $L\oplus K \cong L' \oplus K'$ for some even unimodular lattices $K$ and $K'$. 
%Since the rank of an even unimodular is divisible by $8$,
% $\sign q(L):=\sign L \bmod 8$ is well-defined.  
%Let $M \hookrightarrow L$ be a primitive embedding of non-degenerate even lattices and suppose that $L$ is unimodular, 
% then there is a natural isomorphism $(A(M),q(M))\cong (A(M^\perp),-q(M^\perp))$.  
The genus of $L$ is defined as the set of isomorphism classes of lattices $L'$ such that the signature of $L'$ is the same as that of $L$ and $q(L)\cong q(L')$.
(We sometimes write $q(L)\cong q(L')$ instead of $(A(L),q(L)) \cong (A(L'),q(L'))$.)

\begin{Thm}[\cite{Ni2,Om}] \label{Nik Genus}
Let $L$ be a non-degenerate even lattice.
% with $\rk L \geq 3$.
% $\rk L \ge l(A(L))+2$, where $l(A(L))$ is the minimum number of generators of $A(L)$. 
If $L\cong U(n)\oplus L'$ for a positive integer $n$
 and a lattice $L'$,
 then the genus of $L$ consists of only one class. 
\end{Thm}

%\begin{Thm}[\cite{Ni2}] \label{Nik Genus2}
%Let $L$ be a non-degenerate indefinite even lattice with $\rk L \ge l(A(L))+2$, where $l(A(L))$ is the minimum number of generators of $A(L)$.  
%Then the genus of $L$ contains only one class. 
%\end{Thm}

Let $L$ be an even lattice and $M$ a module such that $L \subset M \subset L^{\vee}$. 
We say that $M$ equipped with the induced bilinear form $\langle *,**\rangle$ is an overlattice of $L$ if $\langle *,**\rangle$ takes integer values on $M$.
Suppose that $M$ is an overlattice of $L$.
%Any lattice which includes $L$ as a sublattice of finite index is considered as an overlattice of $L$. 
Then the subgroup $W:=M/L\subset A(L)$ is isotropic, that is, the restriction of $q(L)$ to $W$ is zero.
There is a natural isomorphism $A(M) \cong W^\bot/W$, where
\begin{equation*}
 W^\bot = \{ x + L \in A(L) \bigm| \langle x , y \rangle \equiv 0 \bmod \Z ~
 (\forall y\in M) \}.
\end{equation*}

%\begin{Prop}[\cite{Ni2}] \label{Nik Isotropic}
%Let $L$ be a non-degenerate even lattice and $M$ a submodule of $L^{\vee}$ such that $L \subset M$. 
%%For a $\Z$-module $M$ such that $L \subset M \subset L^{\vee}$, 
%Then $M$ is an even overlattice of $L$ if and only if the image of $M$ in $A(L)$ is an isotropic subgroup, that is, the restriction of $q(L)$ to $M/L$ is zero. 
%Moreover, there is a natural one-to-one correspondence between the set of even overlattices of $L$ and the set of isotropic subgroups of $A(L)$. 
%\end{Prop}

\begin{Prop}[\cite{Ni2}] \label{Nik Disc form}
Let $K$ and $L$ be non-degenerate even lattices. 
There exists a primitive embedding of $K$ into an even unimodular lattice $\Gamma$ such that $K^{\bot} \cong L$, 
if and only if $(A(K),q(K)) \cong (A(L),-q(L))$. 
More precisely, any such $\Gamma$ is of the form $\Gamma_{\lambda} \subset K^{\vee}\oplus L^{\vee}$ 
for some isomorphism $\lambda\colon(A(K),q(K)) \rightarrow (A(L),-q(L))$, 
where $\Gamma_{\lambda}$ is the overlattice of $K\oplus L$ corresponding to the isotropic subgroup
% (with respect to $q(K)\oplus q(L)$)
$$
\{(x,\lambda(x))\in A(K)\oplus A(L) \bigm| x \in A(K) \} \subset A(K) \oplus A(L). 
$$
%For two isomorphisms $\lambda,\lambda^{'}$, the isomorphism $\phi \oplus \psi$ where $\phi \in \mathrm{O}(K)$ and $\psi \in \mathrm{O}(L)$ 
%can be extended to an isomorphism $\Gamma_{\lambda}\rightarrow \Gamma_{\lambda^{'}}$ if and only if 
%$\lambda^{'}\circ \overline{\phi}=\overline{\psi}\circ \lambda \in \mathrm{O}(A(L),q(L))$.  
\end{Prop}

%\begin{Lem}[{\cite[Lemma 2.5]{HK}}] \label{Involution}
%Let $L$ be a non-degenerate lattice and
% $\iota\in O(L)$ an involution.   
%Then
%% $L/(L^\iota\oplus L_\iota )$ is $2$-elementary,
%% that is,
% $L/(L^\iota \oplus L_\iota) \cong (\Z/2\Z)^n$
% for some $n \le \min\{\rk L^\iota, \rk L_\iota\}$. 
%%Conversely, for a non-degenerate primitive sublattice $M\subset L$ such that $L/(M\oplus M_L^\perp)$ is $2$-elementary, 
%%there exists an involution $\iota\in\mathrm{O}(L)$ with $L^\iota=M$. 
%\end{Lem}

%%%%%%%%%%%%%%%%%%%%%%%%%%%%%%%%%%%%%%%%%%%%%%%%%%%%%%%%%%%%%%%%%%%%%%%%%%%%%%%%%%%%%%%%%%%%%%%%%%%%%%%%%%%%%

\subsection{K3 Surfaces} 
A K3 surface $S$ is a simply-connected compact complex surface $S$ with trivial canonical bundle $\varOmega_{S}^{2}\cong \mathcal{O}_{S}$.
Then $H^2(S,\Z)$ with the cup product is isomorphic to the K3 lattice $U^{\oplus 3} \oplus E_{8} (-1)^{\oplus 2}$. 
It is also endowed with a weight-two Hodge structure
$$
H^2(S,\C)=
H^2(S,\Z)\otimes\C=H^{2,0}(S)\oplus H^{1,1}(S)\oplus H^{0,2}(S).
$$
Let $\omega_S$ be a nowhere vanishing holomorphic 2-form on $S$. 
The space $H^{2,0}(S)\cong\C$ is generated by the class of $\omega_S$, which we denote by the same $\omega_S$. 
%The weight-two Hodge structure on the lattice $H^2(S, \Z)$ is determined by the line $\C\omega_S \subset H^2(S,\C)$. 
%A Hodge isometry is an isometry between 2 lattices with Hodge structures which preserves  the Hodge decompositions. 
The algebraic lattice $NS(S)$ and the transcendental lattice $T(S)$ of $S$ are primitive sublattices of $H^2(S,\Z)$ defined by
$$
NS(S):=\{ x\in H^2(S,\Z) \bigm| \langle x,\omega_S \rangle=0 \},\ \ \  T(S):=NS(S)^\bot_{H^2(S,\Z)}.
$$
Here we extend the bilinear form $\langle *,** \rangle$ on $H^2(S,\Z)$ to that on $H^2(S,\Z)\otimes\C$ linearly.
Note that $NS(S)$ is naturally isomorphic to the Picard group of $S$.
%The open subset $\mathcal{K}_S\subset H^{1,1}(S,\R):=H^{1,1}(S)\cap H^2(S,\R)$ of classes of K\"ahler forms is called the K\"ahler cone of $S$. 
%Set  
%\begin{align}
%\Delta^+&:=\{x \in NS(S) \ | \ x^2=-2, \ x \  \text{effective}\}, \notag \\
% H_{\delta}&:=\{x \in H^2(S,\Z) \ | \ \langle x, \delta \rangle =0\}  \ \ (\delta \in \Delta^+).\notag 
%\end{align}
%By the Hodge index theorem, the set $\{ x \in H^{1,1}(S,\Z) \ | \ x^2>0 \}$ consists of two disjoint connected components. 
%Let $C_S$ be the connected component that contains the K\"ahler cone $\mathcal{K}_S$.  
%We call the connected components of $C_S\setminus \bigcup_{\delta\in \Delta^+} H_\delta$ the chambers of $C_S$. 
%Then the K\"ahler cone $\mathcal{K}_S$ is the chamber given by $\{ x \in C_S \ | \ \langle x, \delta \rangle >0 \ \forall \delta \in \Delta^+\}$. 
%The study of automorphims of K3 surfaces is reduced to lattice theory by the Global Torelli theorem and surjectivity of the period map \cite{BHPV}. 

If a group $G$ acts on a K3 surface $S$, the action of $G$ induces a left action on $H^2(S,\Z)$ by
$g\cdot x:=(g^{-1})^* x$ for  $g\in G$ and $x\in H^2(S,\Z)$. 
An automorphism $g$ of $S$ is called symplectic if $g^*\omega_S=\omega_S$. 
An Enriques surface is the quotient of a K3 surface $S$ by a fixed point free involution $\iota$, which we call an Enriques involution. 
Then we have $\iota^* \omega_S=-\omega_S$, i.e. an anti-symplectic involution. 

\begin{Prop}[{\cite[{Section 2.3}]{AN}}] \label{PROP_enriques_involution}
Let $S$ be a K3 surface.
An involution $\iota \in \Aut(S)$ is an Enriques involution if and only if
$$
 H^2(S,\Z)^\iota\cong U(2)\oplus E_8(-2),\quad
 H^2(S,\Z)_\iota\cong U\oplus U(2)\oplus E_8(-2).
$$
\end{Prop}

%In \cite{Do}, Dolgachev formulated the notion of a lattice polarization of a K3 surface. 
% Let $M$ be a lattice of signature $(1, r-1)$ that can be primitively embedded into the K3 lattice $\Lambda$. 
% An $M$-polarized K3 surface is a projective K3 surface $S$, together with a primitive embedding $i : M \hookrightarrow NS(S)$ 
% such that the image $i(M)$ contains an pseudo-ample element of $NS(S)$. 
% A coarse moduli space can be defined for equivalence classes of pairs $(S,i)$ of $M$-polarized K3 surfaces and an appropriate version of the Torelli type theorems hold \cite{Do}.  

%%%%%%%%%%%%%%%%%%%%%%%%%%%%%%%%%%%%%%%%%%%%%%%%%%%%%%%%%%%%%%%%%%%%%%%%%%%%%%%%%%%%%%%%%%%%%%%%%%%%%%%%%%%%%
%%%%%%%%%%%%%%%%%%%%%%%%%%%%%%%%%%%%%%%%%%%%%%%%%%%%%%%%%%%%%%%%%%%%%%%%%%%%%%%%%%%%%%%%%%%%%%%%%%%%%%%%%%%%%

\section{Calabi--Yau Threefolds of Type K} 
%Recall that, in this paper, a Calabi--Yau threefold $X$ is a K\"ahler threefold $X$ 
%with trivial canonical bundle $\varOmega_{X}^{3}\cong \mathcal{O}_{X}$ and $H^{1}(X, \mathcal{O}_{X})=0$. 
%We {\it do not} assume that $X$ is simply-connected.  
%A reason for this is that simply-connected Calabi--Yau threefolds are not closed under mirror symmetry; 
%a mirror partner of a simply-connected Calabi--Yau threefold may not be simply-connected \cite{GP}.  
%In view of mirror symmetry, our definition seems more natural than {\it strict Calabi--Yau threefolds}, for which we require simply-connectedness. 
%In this paper, we are interested in a certain class of Calabi--Yau threefolds with infinite fundamental group. 

\subsection{Classification and Construction}
\label{SECT_construction}
We begin with a review of \cite{OS, HK}.  
By the Bogomolov decomposition theorem,
a Calabi--Yau threefold $X$ with infinite fundamental group admits an \'{e}tale Galois covering 
either by an abelian threefold or by the product of a K3 surface and an elliptic curve.   
We call $X$ of type A in the former case and of type K in the latter case. 
Among many candidates of such coverings, we can always find a unique smallest one, up to isomorphism as a covering, and we call it the minimal splitting covering.  
The full classification of Calabi--Yau threefolds with infinite fundamental group was completed in \cite{HK}. 
Here we focus on type K: 

\begin{Thm}[\cite{HK}] 
There exist exactly eight Calabi--Yau threefolds of type K, up to deformation equivalence.
The equivalence class is uniquely determined by the Galois group of the minimal splitting covering, 
which is isomorphic to one of the following combinations of cyclic and dihedral groups
$$
C_{2},\ C_2 \times C_2, \ C_2 \times C_2 \times C_2, \ D_{6}, \ D_{8}, \ D_{10}, \ D_{12}, \ \text{and} \ \ C_2 \times D_8. 
%\footnote{We use the symbol $\oplus$ to denote direct products of groups.}.
$$ 
\end{Thm}

We now briefly summarize the construction in \cite{HK}.  
%For a thorough treatment of the construction, we refer the reader to the original paper \cite{HK}. 
Let $X$ be a Calabi--Yau threefold of type K and $\pi \colon S \times E \rightarrow X$ its minimal splitting covering with Galois group $G$. 
There exists a canonical isomorphism $\Aut(S\times E)\cong \Aut(S)\times \Aut(E)$, which induces a faithful $G$-action on each $S$ and $E$:
\[\xymatrix{
\Aut(S) \ar@{}[d]|\bigcup & \ar[l]_{p_1}\Aut(S \times E) \ar@{}[d]|\bigcup \ar[r]^{p_2} & \Aut(E) \ar@{}[d]|\bigcup \\
p_1(G) & \ar[l]^{p_1|_{G}}_{\cong} G \ar[r]_{p_2|_{G}}^{\cong} &  p_2(G) . 
}\]

\begin{Prop}[\cite{OS, HK}] \label{PROP_OS}
%Let $G$ be a Calabi--Yau group and $S\times E$ its target threefold. 
Let $H:=\Ker(G\rightarrow \GL(H^{2,0}(S)))$ and take any $\iota \in G \setminus H$. 
Then the following hold:
\begin{enumerate}
\item $\ord(\iota)=2$ and $G=H\rtimes \langle \iota \rangle$, where the semi-direct product structure is given by $\iota h\iota=h^{-1}$ for all $h\in H$;
\item $g$ acts on $S$ as an Enriques involution if $g\in G\setminus H$; 
\item $\iota$ acts on $E$ as $-1_{E}$ and $H$ as translations of the form $\langle t_{a}\rangle \times \langle t_{b}\rangle\cong C_{n}\times C_{m}$ under an appropriate origin of $E$. 
Here $t_a$ and $t_b$ are translations of order $n$ and $m$ respectively for some $(n,m)\in \{(1,k)(1\le k \le 6), \ (2,2), \ (2,4)\}$. 
\end{enumerate}
\end{Prop}

Conversely, such a $G$-action yields a Calabi--Yau threefold $X:=(S\times E)/G$ of type K. 
Proposition \ref{PROP_OS} provides us with a complete understanding of the $G$-action on $E$, and 
therefore the classification essentially reduces to that of K3 surfaces equipped with actions described in Proposition \ref{PROP_OS}.  

\begin{Def} \label{def_cyg_k3}
Let $G$ be a finite group. 
We say that an action of $G$ on a K3 surface $S$ is Calabi--Yau if the following hold:
\begin{enumerate}
\item $G=H\rtimes \langle \iota \rangle$ with $H\cong C_n\times C_m$ for some $(n,m)\in \{(1,k)\ (1\le k \le 6), \ (2,2), \ (2,4)\}$, and $\ord(\iota)=2$.  
The semi-direct product structure is given by $\iota h\iota=h^{-1}$ for all $h\in H$; 
\item $H$ acts on $S$ symplectically, and any $g\in G\setminus H$ acts as an Enriques involution.
\end{enumerate}
\end{Def}
In what follows, $G$ is always one of the finite groups listed above. 
A basic example of a K3 surface which the reader could bear in mind is the following Horikawa model. 

\begin{Prop}[{Horikawa model \cite[Section V, 23]{BHPV}}] \label{K3->P1P1} 
The double covering $S\rightarrow \mathbb{P}^{1}\times \mathbb{P}^{1}$ branching along a bidegree $(4,4)$-divisor $B$ is a K3 surface if it is smooth. 
We denote by $\theta$ the covering involution on $S$. 
Assume that $B$ is invariant under the involution $\lambda$ of $\PP^1 \times \PP^1$ given by $(x,y)\mapsto (-x,-y)$, where $x$ and $y$ are the inhomogeneous coordinates of $\PP^1\times \PP^1$.  
The involution $\lambda$ lifts to a symplectic involution of $S$. 
Then $\theta \circ \lambda$ is an involution of $S$ without fixed points unless $B$ passed through one of the four fixed points of $\lambda$ on $\PP^1 \times \PP^1$. 
The quotient surface $T=S/\langle \theta \circ \lambda \rangle$ is therefore an Enriques surface. 
\[\xymatrix{
 \ar[d]_{/\langle \theta \rangle } S \ar[r]^{\id} & S  \ar[d]^{/ \langle \theta \circ \lambda \rangle }\\
 \mathbb{P}^{1}\times \mathbb{P}^{1} & T \\
}\]
\end{Prop}

The classical theory of Enriques surfaces says that any generic K3 surface with an Enriques involution is realized as a Horikawa model (\cite[Propositions 18.1, 18.2]{BHPV}). 

\begin{Ex}[Enriques Calabi--Yau threefold] \label{Enr CY3}
Let $S$ be a K3 surface with an Enriques involution $\iota$ and $E$ an elliptic curve with negation $-1_E$.
The free quotient $X:=(S\times E)/\langle (\iota, -1_E) \rangle$ is the simplest Calabi--Yau threefold of type K with Galois group $G=C_2$, known as the Enriques Calabi--Yau threefold.
\end{Ex}

We saw that the Horikawa model gives rise to the Enriques Calabi--Yau threefold.  
In order to obtain other Calabi--Yau threefolds of type K, we consider special classes of Horikawa models as follows. 
We introduce a group $G_0:=H\rtimes C_2$ which is isomorphic to $G$ in Definition \ref{def_cyg_k3} as an abstract group and we denote by $\lambda$ the generator of the second factor $C_2$. 
Let $\rho_{1},\rho_{2} \colon G_0\rightarrow \PGL(2,\C)$ be 2-dimensional complex projective representations.   
We thus get a $G_0$-action $\rho_{1} \times\rho_{2}$ on $\PP^1 \times \PP^1$. 
Suppose that there exists a $G_0$-stable smooth curve $B$ of bidegree $(4,4)$. 
We then obtain a Horikawa K3 surface $S$ as the double covering $\pi \colon S\rightarrow \PP^1 \times \PP^1$ branching along $B$ 
and the $G_0$-action on $\PP^1 \times \PP^1$ lifts to $S$ as a symplectic $G_0$-action. 
We further assume that the curve $B$ does not pass through any of fixed points of $g\in G_0 \setminus H$. 
With the same notation as in Proposition \ref{K3->P1P1}, it can be checked that the symplectic $G_0$-action and the covering transformation $\theta$ commute. 
\[\xymatrix{
  & \ar[d]^{/\langle \theta \rangle } S \ar[r]^{\id} & S  \ar[d]^{/ \langle \theta \circ \lambda \rangle }\\
 G_0 \ar@/^6mm/[r]^{\rho_{1} \times\rho_{2}} \ar@/^8mm/[ur]^{\exists symplectic}& \mathbb{P}^{1}\times \mathbb{P}^{1} & T \\
}\]

By twisting $\lambda$ by $\theta$, we obtain a $G$-action on $S$, i.e. 
$$
\Aut(S)\supset G_0\times \langle \theta \rangle \supset H \rtimes \langle \theta \circ \lambda \rangle = G.
$$
It is shown in \cite{HK} that the $G$-action on $S$  is a Calabi--Yau action and a generic K3 surface equipped with a Calabi--Yau action is realized (not necessarily uniquely) in this way. 
To put it another way, there exist projective representations $\rho_{1},\rho_{2} \colon G_0\rightarrow \PGL(2,\C)$ which satisfy all the assumptions mentioned above. 
In this paper, we do not need explicit presentations of Calabi--Yau actions and thus close this section by just providing two examples. 
%We refer the reader to the original paper \cite[Proposition 3.8]{HK} for more details. 

\begin{Ex} \label{Ex $D_{10}$}
Suppose that $G = D_{12} = \langle a,b \bigm| a^{6}=b^{2}=baba=1 \rangle$. 
Let $x,y$ (resp.\ $z,w$) be homogeneous coordinates of the first (resp.\ second) $\PP^1$.
For $i=1,2$, we define $\rho_{i} \colon D_{12}\rightarrow \PGL(2,\C)$ by
$$
a\mapsto  \begin{bmatrix}
                \zeta_{12}^{i}    &  0  \\
                0  &   \zeta_{12}^{12-i} \\
                \end{bmatrix}, \ \ \ 
b\mapsto  \begin{bmatrix}
                0  &  1  \\
                1  & 0   \\
                \end{bmatrix},
$$
where $\zeta_{k}$ denotes a primitive $k$-th root of unity.  
A basis of the space of $D_{12}$-invariant polynomials of bidegree $(4,4)$ are given by $x^{4}z^{4}+y^{4}w^{4}, x^{4}zw^{3}+y^{4}z^{3}w, x^{2}y^{2}z^{2}w^{2}$. 
 A generic linear combination of these cuts out a smooth curve of bidegree $(4,4)$. 
\end{Ex}

\begin{Ex} \label{EX_d8c2}
Suppose that 
$
G = D_{8}\times C_{2}=\langle a,b,c \bigm| a^{4}=b^{2}=baba=1,ac=ca,bc=cb \rangle
$.
For $i=1,2$, we define $\rho_{i} \colon D_{8}\times C_{2}\rightarrow \PGL(2,\C)$ by
$$
a\mapsto  \begin{bmatrix}
                \zeta_{8}    &  0  \\
                0  &   \zeta_{8}^{7} \\
                \end{bmatrix}, \ \ \ 
b\mapsto  \begin{bmatrix}
                0  &  1  \\
                1  & 0   \\
                \end{bmatrix}, \ \ \
c\mapsto  \begin{bmatrix}
                \sqrt{-1}^{i-1}  &  0  \\
                0  & \sqrt{-1}^{1-i}   \\
                \end{bmatrix}. 
$$
A basis of the space of $D_{8}\times C_{2}$-invariant polynomials of bidegree $(4,4)$ are given by $x^{4}z^{4}+y^{4}w^{4}, x^{4}w^{4}+y^{4}z^{4}, x^{2}y^{2}z^{2}w^{2}$. 
 A generic linear combination of these cuts out a smooth curve of bidegree $(4,4)$. 
\end{Ex}

%%%%%%%%%%%%%%%%%%%%%%%%%%%%%%%%%%%%%%%%%%%%%%%%%%%%%%%%%%%%%%%%%%%%%%%%%%%%%%%%%%%%%%%%%%%%%%%%%%%%%%%%%%%%%

\subsection{Lattices $H^{2}(S,\Z)^{G}$ and $H^2(S,\Z)_{C_2}^H$}
%\subsection{Lattices $M_G$ and $N_G$}

In what follows, we fix the decomposition $G=H \rtimes C_2$ with $C_2=\langle \iota \rangle$ as in Proposition \ref{PROP_OS}.
We define two sublattices of $H^2(S,\Z)$ by $M_G:=H^{2}(S,\Z)^{G}$ and $N_G:=H^2(S,\Z)_{C_2}^H$. 
The lattices $M_G$ and $N_G$ depend only on $G$ (as an abstract group) and a generic K3 surface $S$ with a Calabi--Yau $G$-action has transcendental lattice $N_G$ (see \cite{HK} for details).

\begin{Prop}[cf.\ \cite{Ni,Ha}] \label{symp-inv lattice}
Let $F$ be a group listed in the table below. 
Suppose that $F$ acts on a K3 surface $S$ faithfully and symplectically. 
The isomorphism class of the invariant lattice $H^{2}(S,\Z)^{F}$ is given by the following.
\begin{center}
 \begin{tabular}{|c|c|c|} \hline 
$F$  & $H^{2}(S,\Z)^{F}$ & $\rank H^{2}(S,\Z)^{F}$ \\ \hline
$C_{2}$ & $U^{\oplus 3}\oplus E_{8}(-2)$ & $14$ \\ \hline
$C_{2} \times C_2$ & $U\oplus U(2)^{\oplus 2}\oplus D_{4}(-2)$ & $10$ \\ \hline
$C_{2} \times C_2 \times C_2$ & $U(2)^{\oplus 3}\oplus \langle -4\rangle^{\oplus 2}$ & $8$ \\ \hline
$C_{3}$ & $U\oplus U(3)^{\oplus 2}\oplus A_{2}(-1)^{\oplus2}$ & $10$ \\ \hline
$C_{4}$ & $U\oplus U(4)^{\oplus 2}\oplus \langle 2\rangle^{\oplus 2}$ & $8$ \\ \hline
$D_{6}$ & $U(3)\oplus A_{2}(2)\oplus A_{2}(-1)^{\oplus2}$ & $8$ \\ \hline
$D_{8}$ & $U\oplus \langle 4\rangle^{\oplus2}\oplus \langle -4\rangle^{\oplus3}$ & $7$ \\ \hline
$C_{5},  D_{10}$ & $U\oplus U(5)^{2}$ & $6$ \\ \hline
$C_{6},  D_{12}$ & $U\oplus U(6)^{2}$ & $6$ \\ \hline
$C_{2}\times C_{4}, C_{2}\times D_{8}$ & $U(2)\oplus \langle 4\rangle^{\oplus2}\oplus \langle -4\rangle^{\oplus2}$ & $6$ \\  \hline
\end{tabular}
 \end{center}
\end{Prop}

%We also recall the following useful lemma. 
%\begin{Lem}[{\cite[Lemma 4.1]{HK}}] \label{Rank}
%Let $S$ be a K3 surface equipped with a Calabi--Yau $G$-action. 
%Define $H:=\Ker(G\rightarrow \GL(H^{2,0}(S)))$ and $r:=\rk H^{2}(S,\Z)^{H}$. 
%Then we have 
%$$ 
%\rk M_G=\frac{r}{2}-1, \ \ \ \rk N_G=\frac{r}{2}+1. 
%$$
%%[for any semi-direct decomposition $G=H\rtimes C_2$.]
%\end{Lem}

The main result in this section is the following classification of the lattices $M_G$ and $N_G$. %(we also list $M_G:=H^{2}(S,\Z)^{G}$ for future reference).   
%The proof of the proposition is quite involved and we need to refer to the proof of the {\it Key Lemma} in our previous paper \cite{HK}. 

\begin{Thm} \label{G-inv lattice}
%Let $S$ be a K3 surface equipped with a Calabi--Yau $G$-action. 
%[For any semi-direct decomposition $G= H\rtimes C_2$, where $H:=\Ker(G\rightarrow \GL(H^{2,0}(S)))$,]
The lattices $M_G$ and $N_G$ for each $G$ are given by the following. 
\begin{center}
% \begin{tabular}{|c|c|c|} \hline 
% $G$  & $M_G$ & $N_G$\\ \hline
% $C_{2}$ & $U(2)\oplus E_{8}(-2)$  & $U\oplus U(2)\oplus E_8(-2)$  \\ \hline
% $C_2 \times C_2$  &  $U(2)\oplus D_4(-2)$  &  $U(2)^{\oplus 2}\oplus D_4(-2)$   \\ \hline
%  $C_2\times C_2 \times C_2$  &  $U(2)\oplus \langle -4 \rangle^{\oplus 2} $  & $U(2)^{\oplus2}\oplus \langle -4 \rangle^{\oplus2}$   \\ \hline
% $D_{6}$ & $U(2)\oplus A_{2}(-2)$  & $U(3)\oplus U(6) \oplus A_2(-2)$ \\ \hline
%$D_{8}$ & $U(2)\oplus \langle -4\rangle$  &  $U(4)^{\oplus2}\oplus \langle -4\rangle$  \\ \hline
%$D_{10}$ & $U(2)$  &  $U(5)\oplus U(10)$   \\ \hline
%$D_{12}$ & $U(2)$  &  $U(6)^{\oplus2}$  \\ \hline
%$C_{2}\times D_{8}$ & $U(2)$  &  $U(4) \oplus \langle 4\rangle\oplus \langle -4\rangle $ \\  \hline
% \end{tabular}
 \begin{tabular}{|c|c|c|c|} \hline 
 $G$  & $H$ & $M_G$ & $N_G$\\ \hline
 $C_{2}$ & $C_1$ & $U(2)\oplus E_{8}(-2)$  & $U\oplus U(2)\oplus E_8(-2)$  \\ \hline
 $C_2 \times C_2$  & $C_2$& $U(2)\oplus D_4(-2)$  &  $U(2)^{\oplus 2}\oplus D_4(-2)$   \\ \hline
  $C_2\times C_2 \times C_2$  & $C_2 \times C_2$ & $U(2)\oplus \langle -4 \rangle^{\oplus 2} $  & $U(2)^{\oplus2}\oplus \langle -4 \rangle^{\oplus2}$   \\ \hline
 $D_{6}$ & $C_3$ & $U(2)\oplus A_{2}(-2)$  & $U(3)\oplus U(6) \oplus A_2(-2)$ \\ \hline
$D_{8}$ & $C_4$ & $U(2)\oplus \langle -4\rangle$  &  $U(4)^{\oplus2}\oplus \langle -4\rangle$  \\ \hline
$D_{10}$ & $C_5$ & $U(2)$  &  $U(5)\oplus U(10)$   \\ \hline
$D_{12}$ & $C_6$ & $U(2)$  &  $U(6)^{\oplus2}$  \\ \hline
$C_{2}\times D_{8}$ & $C_2 \times C_4$ & $U(2)$  &  $U(4) \oplus \langle 4\rangle\oplus \langle -4\rangle $ \\  \hline
 \end{tabular}
 \end{center}
\end{Thm}
\begin{Rem} \label{REM_Q_isom}
From the table in Theorem \ref{G-inv lattice}, one can check that there exists an isomorphism $N_G\otimes\Q\cong (U\oplus M_G) \otimes \Q$ as quadratic spaces over $\Q$.
\end{Rem}
\begin{proof}[Proof of Theorem \ref{G-inv lattice}]
It suffices to determine $N_G$ because the classification of $M_G$ was completed in \cite[Proposition 4.4]{HK}. 
We set $\Lambda:=H^{2}(S,\Z)$ and employ the lattice theory reviewed in Section 2.
In \cite[Proposition 3.11]{HK}, we gave a projective model of a generic K3 surface $S$ with a Calabi--Yau $G$-action as the double covering $S \rightarrow \PP^1\times\PP^1$ branching along a smooth curve $C$ of bidegree $(4,4)$ (see Section \ref{SECT_construction}).
%Moreover, the transcendental lattice of such $S$ coincides with $N_G$.
Let $\theta$ denote the covering transformation of $S\rightarrow \PP^1\times \PP^1$. 
We define $G'$ to be the group generated by $H$ and $\iota\theta$.
%where $\iota$ is the generator of $C_2$ corresponding to an Enriques involution. 
Then $G'$ is isomorphic to $G$ as an abstract group, the action of $G'$ on $S$ is symplectic, and we have $N_G=\Lambda^{G'}_\iota$. 
Let $L\subset \Lambda$ denote the pullback of $H^2(\PP^1\times \PP^1,\Z)$, which is isomorphic to $U(2)$.
Let $e,f$ denote the standard basis of $L$, that is, $e^2=f^2=0$ and $\langle e,f \rangle=2$.
For our projective model, $G$ acts on $L$ trivially.
Since $S/\langle \theta \rangle \cong \PP^1 \times \PP^1$,
% and the action of $G$ preserves each ruling $\PP^1 \times \PP^1 \rightarrow \PP^1$ for our projective model,
 we have $L=\Lambda^\theta=\Lambda^{\langle G,\theta \rangle}$.
Hence
\begin{equation*}
 N_G=\Lambda^{G'}_\iota
 =L^\bot_{\Lambda^{G'}}=(\Lambda_{G'}\oplus L)^\bot_\Lambda.
\end{equation*}
Recall that $\Lambda^{G'}$ is given in Proposition \ref{symp-inv lattice}. 

 \underline{Case $H= C_1$.} By Proposition \ref{PROP_enriques_involution}, we have $N_G=U\oplus U(2)\oplus E_8(-2)$.
 
% \underline{Case $H\cong C_2$.} (Need a simpler proof)
 
  \underline{Case $H= C_2$.}
%By the result of Ito and Ohashi \cite[Theorem 1.1]{IO}, we have $N_G\cong U(2)^{\oplus2}\oplus D_4(-2) =: \Gamma$. 
%We give a proof of this fact for the sake of completeness.
Set $\Gamma:= U(2)^{\oplus2}\oplus D_4(-2)$.
We have $\Lambda^{G'}\cong U\oplus \Gamma$.
%We have $\Lambda^{G'}\cong U\oplus U(2)^{\oplus 2} \oplus D_4(-2)$.
For our projective model, the branching curve $C$ is defined by a (generic) linear combination of polynomials of the form
$$
 x^i y^{4-i} z^j w^{4-j} + x^{4-i} y^i z^{4-j} w^j, \quad
 0\leq i,j \leq 4, \ i \equiv 0 \ \bmod 2,
$$
 where $x,y$ (resp.\ $z,w$) are homogeneous coordinates of the first (resp.\ second) $\PP^1$.
Note that the equation $z=0$ defines a reducible curve on $S$, which has the two components $D_1,D_2$ such that $D_i^2=-2$ and $\langle D_1,D_2 \rangle=2$.
Hence we may assume that the class of $D_1$ or $-D_1$ is given as  $u:=(v+e)/2 \in \Lambda$ for some $v\in \Lambda_{G'}$ by interchanging $e$ and $f$ if necessary.
We then have the following inclusions:
$$
 K_0:= \Lambda_{G'}\oplus L \subset K_1:=K_0 + \Z u
 \subset (N_G)_\Lambda^\bot.
$$
We show $K_1=(N_G)_\Lambda^\bot$.
%We have $q(K_0)\cong q(\Lambda_{G'}) \oplus q(L)$.
Let $\overline{u}$ denote the class of $u$ in $A(K_0)$.
We have
\begin{equation*} \label{EQ_k0_k1}
 A(K_1)=\overline{u}^\bot/\langle \overline{u} \rangle
 \cong B \cap \overline{u}^\bot, \quad
 B:= \left( (\Lambda_{G'})^\vee \oplus \Z (f/2) \right)/K_0\subset A(K_0).
\end{equation*}
Hence the projection $B \rightarrow A(\Lambda_{G'})$ induces the isomorphism $q(K_1) \cong q(\Lambda_{G'})$.
Hence
\begin{equation} \label{EQ_q_gamma}
 q(K_1) \cong q(\Lambda_{G'}) \cong -q(\Lambda^{G'}) \cong -q(\Gamma)
\end{equation}
 by Proposition \ref{Nik Disc form}.
Again, by Proposition \ref{Nik Disc form}, there exists a primitive embedding $\varphi\colon K_1 \rightarrow \Lambda$ 
such that $(\varphi(K_1))_\Lambda^\bot \cong \Gamma$.
By Proposition \ref{symp-inv lattice}, we have $(\Lambda_{G'})^\iota = \Lambda_{\iota \theta} \cong E_8(-2)$.
Thus
\begin{equation*}
 (K_0)^\iota \cong E_8(-2) \oplus U(2) \subset (K_1)^\iota
 \subset ((N_G)_\Lambda^\bot)^\iota \cong E_8(-2) \oplus U(2)
\end{equation*}
by Proposition \ref{PROP_enriques_involution}.
Hence $(K_1)^\iota \cong E_8(-2) \oplus U(2)$.
Since $H$ acts on $A(K_1)$ trivially and $A((K_1)^\iota)$ is a $2$-elementary group, the action of $G$ on $K'_1:=\varphi(K_1)$ (via $\varphi$) extends to that on $\Lambda$ in such a way that $H$ (resp.\ $\iota$) acts on $(K'_1)_\Lambda^\bot$ trivially (resp.\ as negation).
This action of $G$ on $\Lambda$ gives a Calabi--Yau $G$-action on a K3 surface (see \cite[Section 3.3]{HK} for details).
By the uniqueness of a Calabi--Yau $G$-action \cite[Theorem 3.19]{HK}, we have $K_1=(N_G)_\Lambda^\bot$.
Hence $q(N_G) \cong -q(K_1) \cong q(\Lambda^{G'}) \cong q(\Gamma)$ by (\ref{EQ_q_gamma}) and Proposition \ref{Nik Disc form}.
Therefore $N_G \cong \Gamma$ by Theorem \ref{Nik Genus}.

\underline{Case $H= C_2 \times C_2$.}
%Since $S$ is realized as a Horikawa model \cite{HK}, we may assume, after suitable twists by the covering transformation, 
%that $C_2 \times C_2 \times C_2$ acts on $S$ symplectically. 
%By Theorem \ref{symp-inv lattice}, we see that $H^2(S,\Z)^{C_2^{\oplus 3}}\cong U(2)^{\oplus 3}\oplus \langle -4\rangle^{\oplus 2}$. 
%
Since $\Lambda^{G'} \cong U(2)^{\oplus 3}\oplus \langle -4\rangle^{\oplus 2}$, we have $N_G = L^\bot_{\Lambda^{G'}} \cong U(2)^{\oplus 2}\oplus \langle -4\rangle^{\oplus 2}$ by Theorem \ref{Nik Genus}.
%N_G=(U(2))^\perp_{H^2(S,\Z)^{C_2^{\oplus 3}}}\cong U(2)^{\oplus 2}\oplus \langle -4\rangle^{\oplus 2} 
%because the covering transformation acts trivially on the polarization $U(2)$ and as negation on its orthogonal complement.  

\underline{Case $H= C_3,C_5$.}
By a similar argument to Case $H = C_2$, we conclude that $q(N_G) \cong q(\Lambda^{G'})\oplus q(L)$.
The lattice $N_G$ is uniquely determined by Theorem \ref{Nik Genus}.
%
%Set $\Gamma:=U(3)\oplus U(6) \oplus A_2(-2)$.
%We have $\Lambda^{G'}\cong U(3)\oplus A_{2}(2)\oplus A_{2}(-1)^{\oplus2}$ and
%$$
% -q(\Lambda_{G'} \oplus L)\cong q(\Lambda^{G'}) \oplus q(L)
% \cong q(\Gamma).
%$$
%By a similar argument to Case $H = C_2$, we conclude that $(N_G)^\bot_\Lambda = \Lambda_{G'} \oplus L$ and that $N_G\cong \Gamma$.
%
%We know that, by Theorem \ref{symp-inv lattice},  
%$$
%M_G\cong U(2)\oplus A_{2}(-2) \subset H^2(S,\Z)^{C_3}\cong U\oplus U(3)^{\oplus2}\oplus A_2(-1)^{\oplus2}. 
%$$ 
%By localizing at $p=2,3$, we see that the discriminant group $A(N_G)$ is given by $A(N_G)\cong \Z_2^{\oplus4}\oplus \Z_3^{\oplus3}$. 
%Proposition \ref{Nik Genus2} then implies that the genus $g_{N_G}$ is a one-point set and we see that $U(3)\oplus U(6) \oplus A_2(-2)$. 

\underline{Case $H= C_4,C_6$.}
%We have $\Lambda^{G'}\cong U\oplus \langle 4\rangle^{\oplus2}\oplus \langle -4\rangle^{\oplus3}$.
For our projective model, the branching curve $C$ is defined by a (generic) linear combination of the following polynomials:
\begin{gather*}
 x^4 z^3 w+y^4 z w^3,x^4 z w^3+y^4 z^3 w,x^2 y^2 z^4+x^2 y^2 w^4,  x^2 y^2 z^2 w^2
 \quad \text{if~} H=C_4, \\
 x^4 z^4+y^4 w^4,x^4 z w^3+y^4 z^3 w,x^2 y^2 z^2 w^2
 \quad \text{if~} H=C_6.
\end{gather*}
Note that the curve on $S$ defined by $z=0$ is reducible in each case.
Similarly to Case $H = C_2$, we conclude that $(N_G)_\Lambda^\bot$ contains $\Lambda_{G'}\oplus L$ as a sublattice of index $2$ and that $q(N_G)\cong q(\Lambda^{G'})$.
The lattice $N_G$ is uniquely determined by Theorem \ref{Nik Genus}.
For Case $H=C_4$, we use the isomorphism $\langle 1 \rangle \oplus \langle -1 \rangle^{\oplus 2} \cong U\oplus \langle -1 \rangle$.
%
%we find that $N_G\cong \langle 4\rangle^{\oplus2}\oplus \langle -4\rangle^{\oplus3} \cong U(4)^{\oplus 2} \oplus \langle -4\rangle$ by Theorem \ref{Nik Genus}.
%Here we use the isomorphism $\langle 1 \rangle \oplus \langle -1 \rangle^{\oplus 2} \cong U\oplus \langle -1 \rangle $.
%
%By \cite[Porposition 3.11]{HK}, the K3 surface $S$ is realized as a $(2,2,2)$-surface in $\PP^1\times \PP^1\times \PP^1$. 
%The generators of $M_G$ are restriction of the hyperplane class in each $\PP^1$ and we easily see that $M_G \cong U(2)\oplus \langle -4\rangle$. 
%
%By a similar argument to Case $H\cong C_2$,

%\underline{Case $H= C_5$.}
%The lattice $N_G$ is determined in the same manner as in Case $H= C_3$. 

%\underline{Case $H= C_6$.}
%We have $\Lambda^{G'}\cong U\oplus U(6)^{\oplus 2}$.
%The branching curve $C$ is defined by a (generic) linear combination of the following polynomials:
%\begin{equation*}
% x^4 z^4+y^4 w^4,x^4 z w^3+y^4 z^3 w,x^2 y^2 z^2 w^2.
%\end{equation*}
%Note that the curve on $S$ defined by $z=0$ is reducible.
%Similarly to Case $H = C_2$, we conclude that $(N_G)_\Lambda^\bot$ contains $\Lambda_{G'}\oplus L$ as a sublattice of index $2$ and that $q(N_G)\cong q(\Lambda^{G'})$.
%Hence we have $N_G\cong U(6)^{\oplus 2}$ by Theorem \ref{Nik Genus}.
%
%For the cases where $N_G$ has not been determined, we apply Proposition \ref{Nik Genus} and conclude that the genus $g_{N_G}$ is a one-point set. 
%The representative of $g_{N_G}$ is given in the above table. 

\underline{Case $H= C_2 \times C_4$.}
Since $\Lambda^{G'} \cong U(2) \oplus \langle 4\rangle^{\oplus 2} \oplus \langle -4\rangle^{\oplus 2}$, we have $N_G = L^\bot_{\Lambda^{G'}} \cong \langle 4\rangle^{\oplus2}\oplus \langle -4\rangle^{\oplus2} \cong U(4) \oplus \langle 4\rangle\oplus \langle -4\rangle$ by Theorem \ref{Nik Genus}.
%A similar argument to Case $H\cong C_2 \times C_2$ works. 
%We note that $\cong U(4) \oplus \langle 4\rangle\oplus \langle -4\rangle$. 
\end{proof}

%%%%%%%%%%%%%%%%%%%%%%%%%%%%%%%%%%%%%%%%%%%%%%%%%%%%%%%%%%%%%%%%%%%%%%%%%%%%%%%%%%%%%%%%%%%%%%%%%%%%%%%%%%%%%
%%%%%%%%%%%%%%%%%%%%%%%%%%%%%%%%%%%%%%%%%%%%%%%%%%%%%%%%%%%%%%%%%%%%%%%%%%%%%%%%%%%%%%%%%%%%%%%%%%%%%%%%%%%%%

\subsection{Brauer Groups $\mathrm{Br}(X)$}
%\subsection{Some Topological Computation}
Throughout this section, $X$ denotes a Calabi--Yau threefold of type K and $\pi \colon S\times E \rightarrow X$ its minimal splitting covering with Galois group $G=H\rtimes \langle \iota \rangle$. 
%We often use the notation $\Z_k:=\Z/k\Z$ instead of $C_k$ when it is regarded as a $\Z$-module rather than a group.  

\begin{Def}
The Brauer group $\mathrm{Br}(Y)$ of a smooth projective variety $Y$ is defined by $\mathrm{Br}(Y):=\mathrm{Tor}(H^2(Y,\mathcal{O}_Y^{\times}))$. 
Here $\mathrm{Tor}(-)$ denotes the torsion part.
\end{Def}
By the exact sequence $H^2(Y,\mathcal{O}_Y) \rightarrow H^2(Y,\mathcal{O}_Y^{\times}) \rightarrow H^3(Y,\Z) \rightarrow H^3(Y,\mathcal{O}_Y) $ 
and the universal coefficient theorem, for a Calabi--Yau threefold $Y$, we have 
$$
\mathrm{Br}(Y)\cong \mathrm{Tor}(H^3(Y,\Z))\cong \mathrm{Tor}(H_2(Y,\Z)) \cong \mathrm{Tor}(H^4(Y,\Z)). 
$$
Therefore the Brauer group of a Calabi--Yau threefold is topological, in contrast to that of a K3 surface, which is analytic.

%\begin{Prop}\label{non-trivial Brauer}
%The Brauer group $\mathrm{Br}(X)$ is non-trivial if $|H|$ is odd. 
%\end{Prop}
%
%\begin{proof}
%The minimal resolution $\widetilde{S/G}$ of the quotient surface $S/G$ is an Enriques surface. 
%Moreover, a standard Mayer--Vietoris sequence argument shows that 
%$$
%H_2(\widetilde{S/G},\Z)/(\oplus_{i=1}^j \Z E_i) \cong H_2(S/G,\Z),
%$$ 
%where $E_1,\dots,E_j$ are the exceptional divisors of the minimal resolution. 
%Therefore we conclude that $\mathrm{Tor}(H_2(S/G,\Z))$ contains an element of order $2$.  
%The natural projection $\pi:X\rightarrow S/G$ admits an $|H|$-section $s:S/G\rightarrow X$.  
%This shows the existence of an element of order $2$ in $H_2(X,\Z)$ when $|H|$ is odd as we have $\pi_*\circ s_*=|H| \id$. 
%\end{proof}

%We may also investigate $\mathrm{Br}(X)$ as follows. 
The projection $E \times S \rightarrow E$ descends to the locally isotrivial K3-fibration $f\colon X \rightarrow B:=E/G \cong \PP^1$ with $4$ singular fibers.
Each singular fiber $f^{-1}(p_i)$ ($p_i\in B$, $1 \le i \le4$) is a (doubled) Enriques surface.
%Let $\pi:X\rightarrow B:=E/G \cong \PP^1$ denote the natural isotrivial K3-fibration with $4$ singular fibers. 
%Namely, for $p_i\in B$ ($1 \le i \le4$), which is the image of each fixed point of the action of $\iota$ on $E$, the fiber $\pi^{-1}(p_i)$ is a (doubled) Enriques surface.
%each fixed points of the Enriques involution $p_i\in \PP^1$ ($1 \le i \le4$),
Let $U_i$ be a neighborhood of $p_i$ isomorphic to a disk such that $U_i \cap U_j=\emptyset$ for $i \ne j$.
We use the following notations: $U_i^*=U_i\setminus \{ p_i \}$, $\overline{U_i}=f^{-1}(U_i)$,
$\overline{U_i^*}=f^{-1}(U_i^*)$, $B^*=B\setminus \{p_i\}_{i=1}^4$ and $\overline{B^*}=f^{-1}(B^*)$. 
Since $S/\langle \iota \rangle \hookrightarrow \overline{U_i}$ is a deformation retraction, we have\footnote{For the homology groups of Enriques surfaces, see e.g.\ \cite[Section 9.3]{DIK}.}
$$
 H_2(\overline{U_i},\Z)\cong H_2(S/\langle \iota \rangle,\Z) \cong \Z^{\oplus 10}\oplus \Z_2. 
$$
Moreover, by considering monodromies, we obtain\footnote{For the proof of the isomorphism $H_2(S,\Z)/  \{ x-\iota x \bigm| x\in H_2(S,\Z) \} \cong \Z^{\oplus 10}\oplus (\Z_2)^{\oplus 2}$, see Lemma \ref{LEM_inv_disc}.}
$$
H_2(\overline{U_i^*},\Z)\cong H_2(S,\Z)/  \{ x-\iota x \bigm| x\in H_2(S,\Z) \} \cong \Z^{\oplus 10}\oplus (\Z_2)^{\oplus 2}
$$
and 
\begin{equation*}
 H_2(\overline{B^*},\Z)\cong H_2(S,\Z)/ \langle x-g x \bigm| x\in H_2(S,\Z), ~ g\in G \rangle. 
\end{equation*}
The Mayer--Vietoris sequence reads
\begin{equation} \label{EQ_Mayer-Vietoris}
 \bigoplus_{i=1}^4 H_2(\overline{U_i^*},\Z)  \mathop{\longrightarrow}^{\alpha}
 H_2(\overline{B^*},\Z) \oplus \left(\bigoplus_{i=1}^4 H_2(\overline{U_i},\Z) \right)
 \mathop{\longrightarrow} H_2(X,\Z)  \longrightarrow \bigoplus_{i=1}^4 H_1(\overline{U_i^*},\Z).
\end{equation}
Since $H_1(\overline{U_i^*},\Z) \cong \Z$, 
 we conclude that $\mathrm{Br}(X) \cong \mathrm{Tor}(P)$, where
\begin{equation} \label{EQ_P_and_Q}
 P=\frac
  {H_2(\overline{B^*},\Z) \oplus \left(\bigoplus_{i=1}^4 H_2(\overline{U_i},\Z)\right)}
  {\alpha(\bigoplus_{i=1}^4 H_2(\overline{U_i^*},\Z))}
 \cong \frac{ H_2(\overline{B^*},\Z) }{Q}.
\end{equation}
Here $Q$ denotes the image in $H_2(\overline{B^*},\Z)$ of the kernel of the natural map
 $H_2(\overline{U_i^*},\Z) \rightarrow H_2(\overline{U_i},\Z)$,
 which is surjective \cite[Section 9.3]{DIK}.
Note that $Q$ does not depend on the index $i$ and, by Lemma \ref{LEM_Q_is_Z2} below, we have $Q\cong \Z_2$.
%[For example, for $G=C_2$ we have $H_2(\overline{B^*},\Z)\cong \Z^{10} \oplus (\Z_2)^{\oplus 2}$, thus $\mathrm{Br}(X)\cong \Z_2$. 
%We expect that a detailed analysis in each case shows that $\mathrm{Br}(X)\cong\Z_2$ for a Calabi--Yau threefold $X$ of type K.]
In what follows, we determine the torsion part of $H_2(\overline{B^*},\Z)$.
We identify $H_2(S,\Z)$ with $\Lambda:=H^2(S,\Z)$ by Poincar\'e duality.

\begin{Lem} \label{LEM_Q_is_Z2}
In the equation (\ref{EQ_P_and_Q}), we have $Q\cong \Z_2$.
\end{Lem}
\begin{proof}
As an example, we consider the case $H \cong C_2\times C_4$.
We can show the assertion for the other cases in a similar manner.
We use the same notation as in Example \ref{EX_d8c2}.
The branching curve $B$ is defined by
\begin{equation*}
 c_1 (x^{4}z^{4}+y^{4}w^{4}) + c_2 (x^{4}w^{4}+y^{4}z^{4})
 + c_3 x^{2}y^{2}z^{2}w^{2} = 0, \quad (c_1,c_2,c_3)\in \C^3.
\end{equation*}
We  consider the degeneration of $S$ to $S_0$,
 where $S_0$ is defined by generic $(c_1,c_2,c_3)$ such that
 $p:=(1:1)\times (1:1) \in B$, that is, $2c_1+2c_2+c_3=0$.
Note that $S_0$ has a singular point of type $A_1$ at the inverse image of $p$.
Let $\gamma$ denote the vanishing cycle corresponding to this singular point.
Then we have $\iota \gamma=-\gamma$ and $\gamma^2=-2$.
Let $\widetilde{S}$ denote the minimal desingularization of $S/H$, which is a K3 surface.
Since the induced action of $\iota$ on $\widetilde{S}$ is an Enriques involution, we have
\begin{equation*}
 \{ x - \iota x \bigm| x \in H_2(\widetilde{S},\Z) \}
 \cong 2 (H_2(\widetilde{S},\Z)_\iota)^\vee \cong U(4)\oplus U(2)\oplus E_8(-2)
\end{equation*}
 by Proposition \ref{PROP_enriques_involution} (cf.\ \cite[Section 4.4.4]{DIK}).
This implies that $(x-\iota x)^2 \equiv 0 \bmod 4$ for any $x\in H_2(\widetilde{S},\Z)$.
Assume that
\begin{equation*}
 \gamma \in \langle x-g x \bigm| x\in H_2(S,\Z), ~ g\in G \rangle=:D.
\end{equation*}
Then the pushforward of $\gamma$ in $H_2(\widetilde{S},\Z)$ is of the form $x-\iota x$ with $x\in H_2(\widetilde{S},\Z)$, which contradicts to $\gamma^2=-2$.
Threfore $\gamma \not\in D$ and $Q$ is generated by $\gamma \bmod D$.
\end{proof}

\begin{Lem} \label{LEM_coinv_gen}
We have $\Lambda_H= \langle x-g x \bigm| x\in \Lambda, ~ g\in H \rangle$. 
%\begin{equation*} \label{EQ_coinv_prim}
%\Lambda_H= \langle x-g x \bigm| x\in \Lambda, ~ g\in H \rangle. 
%\end{equation*}
\end{Lem}
\begin{proof}
First assume that $H$ is a non-trivial cyclic group generated by $h$.
Since $\Lambda$ is unimodular, any element $x\in \Lambda$
 is written as $x'+x''$ with $x' \in (\Lambda_H)^\vee$ and
 $x'' \in (\Lambda^H)^\vee$ (Proposition \ref{Nik Disc form}).
Hence the map $\gamma$ defined by
$$
 \gamma:=\id-h \colon (\Lambda_H)^\vee \rightarrow
 L:=\langle x-g x \bigm| x\in \Lambda, ~ g\in H \rangle
$$
 is an isomorphism.
Note that we have the following inclusions: $L \subset \Lambda_H \subset (\Lambda_H)^\vee$. 
The eigenvalues of the action of $h$ on $\Lambda_H$
 are given as in the following table (see Proposition \ref{symp-inv lattice}).
$$
\begin{array}{|c|c|c|}
 \hline
 H & \text{Eigenvalues} & \det(\gamma) \\
 \hline
 C_2 & (-1)^8 & 2^8 \\ \hline
 C_3 & (\zeta_3)^6(\zeta_3^2)^6 & 3^6 \\ \hline
 C_4 & (-1)^6(\zeta_4)^4(-\zeta_4)^4 & 2^{10} \\ \hline
 C_5 & (\zeta_5)^4(\zeta_5^2)^4(\zeta_5^3)^4(\zeta_5^4)^4 & 5^4 \\ \hline
 C_6 & (-1)^4(\zeta_3)^4(\zeta_3^2)^4(-\zeta_3)^2(-\zeta_3^2)^2
  & 2^4 \cdot 3^4 \\
 \hline
\end{array}
$$
In each case, we have $\det(\gamma)=|\disc(\Lambda_H)|$.
Hence
\begin{equation*}
 |(\Lambda_H)^\vee/L|=\det(\gamma)
 =|\disc(\Lambda_H)|=|(\Lambda_H)^\vee/\Lambda_H|,
\end{equation*}
 which implies $L=\Lambda_H$.

Next we consider the case $H=C_2\times K$ with $K=C_2$ or $C_4$.
We define $E$ and $F$ by
$$
 E:=\Lambda_H/\left( (\Lambda_H)_{K} \oplus (\Lambda_H)^{K} \right), \quad
 F:=\Lambda_{C_2}/\left( (\Lambda_{C_2})_{K} \oplus (\Lambda_{C_2})^{K} \right).
$$
We have
$$
 (\Lambda_H)_K=\Lambda_K, \quad
 \Lambda_K \cap \Lambda_{C_2}=(\Lambda_{C_2})_K, \quad
 M:=(\Lambda_H)^{K}=(\Lambda_{C_2})^{K}.
$$
Hence $F$ is considered as a subgroup of $E$.
Note that $\Lambda_K$ and $\Lambda_{C_2}$ are contained in $L$
 by the argument above.
Therefore, in order to show $\Lambda_H=L$,
 it is enough to show $E=F$.
We have
$$
 |E|^2 \cdot |\disc(\Lambda_H)|=|\disc(\Lambda_K)| \cdot |\disc(M)|.
$$
Since $\Lambda_{C_2}(1/2)$ is isomorphic to the unimodular lattice $E_8(-1)$ by Proposition \ref{symp-inv lattice}, we obtain
$$
 |F|=|\disc(M(1/2))|=|\disc(M)| \cdot 2^{-r}, \quad
 r=\rank M=\rank \Lambda_H-\rank \Lambda_K,
$$
 by Propositions \ref{Nik Disc form}.
We have the following table (see Proposition \ref{symp-inv lattice}).
$$
\begin{array}{|c|c|c|c|c|c|}
 \hline
 K & \disc(\Lambda_H) & \disc(\Lambda_K)
  & \rank \Lambda_H & \rank \Lambda_K & r \\ \hline
 C_2 & 2^{10} & 2^8 & 8 & 12 & 4 \\ \hline
 C_4 & 2^{10} & 2^{10} & 16 & 14 & 2 \\
 \hline
\end{array}
$$
Hence, in each case, we have $|E|^2/|F| =2^r \cdot \disc(\Lambda_K)/\disc(\Lambda_H)=2^2$. 
Note that $M(1/2)$ is not unimodular, that is, $|F|>1$,
 because the rank of any definite even unimodular lattice is divisible by $8$ (see Section \ref{SECT_lattice}).
Since $|E|/|F|=|E/F|$ is an integer, we have $|E|=|F|=2^2$,
 which implies $E=F$.
\end{proof}

\begin{Lem} \label{LEM_torsion_B_star}
The torsion part of
\begin{equation} \label{EQ_B_star}
H_2(\overline{B^*},\Z)
\cong \Lambda / \langle x-g x \bigm| x\in \Lambda, ~ g\in G \rangle %\\
\cong (\Lambda^H)^\vee / \{ x-\iota x \bigm| x\in (\Lambda^H)^\vee \}
\end{equation}
is isomorphic to $\Z_2^{\oplus n}$. 
Here $n$ is given by $n=\rank \Lambda^H_\iota - a$ for $\Lambda^H/( \Lambda^G \oplus \Lambda^H_\iota )\cong\Z_2^{\oplus a}$. 
% |\disc \Lambda^G| \cdot |\disc \Lambda^H_\iota|/|\disc \Lambda^H|
% =2^{2a}.
\end{Lem}
\begin{proof}
The projection $\Lambda\rightarrow (\Lambda^H)^\vee$ induces an isomorphism $\varphi \colon \Lambda / \Lambda_H \rightarrow (\Lambda^H)^\vee$ by Proposition \ref{Nik Disc form}.
Note that the action of $\iota$ preserves $\Lambda^H$.
By Lemma \ref{LEM_coinv_gen}, the second isomorphism in (\ref{EQ_B_star}) is derived from $\varphi$.
The assertion of the lemma follows from Lemma \ref{LEM_inv_disc} below.
\end{proof}

\begin{Lem} \label{LEM_inv_disc}
Let $\Gamma$ be a non-degenerate lattice.
If an involution $\iota$ acts on $\Gamma$ non-trivially, then
\begin{equation*}
 \mathrm{Tor}( \Gamma^\vee / \{ x - \iota x \bigm| x\in \Gamma^\vee \} )
 \cong \Z_2^{\oplus n}, \quad n=\rank \Gamma_\iota - a.
\end{equation*}
Here $a \in \N$ is determined by
 $\Gamma/( \Gamma_\iota \oplus \Gamma^\iota ) \cong \Z_2^{\oplus a}$.
\end{Lem}
\begin{proof}
For a free $\Z$-module $W=\Z w_1 \oplus \Z w_2$ of rank $2$,
 we define an action of $\iota$ on $W$ by
 $\iota(w_1)=w_2$ and $\iota(w_2)=w_1$.
As a $\Z$-module with an $\iota$-action,
 $\Gamma^\vee$ is decomposed into the direct sum
 $\Gamma^\vee\cong V_+ \oplus V_- \oplus W^{\oplus b}$,
 where $\iota$ acts as $\pm 1$ on $V_{\pm}$
 (see \cite[\S 74]{CR}).
One can check that there is a (non-canonical) isomorphism
 $\Gamma \cong \Gamma^\vee = \mathrm{Hom}_\Z(\Gamma,\Z)$ of $\Z$-modules with an $\iota$-action.
Hence $a=b$.
We have $\mathrm{Tor}( \Gamma^\vee / \{ x - \iota x \bigm| x\in \Gamma^\vee \} ) \cong V_-/2 V_-$. 
Since $\rank \Gamma_\iota=\rank V_- + a$, the assertion holds.
\end{proof}

\begin{Thm} \label{Thm: brauer_grp}
We have $\mathrm{Br}(X)\cong \Z_2^{\oplus m}$, where $m$ is given by the following.
\begin{equation*}
\begin{array}{|c|c|c|c|c|c|c|c|c|}
 \hline
 G & C_{2} & C_2 \times C_2 & C_2 \times C_2 \times C_2 & D_{6} & D_{8} & D_{10} & D_{12} & C_2 \times D_8 \\ \hline
 m & 1 & 2 & 3 & 1 & 2 & 1 & 2 & 3 \\ \hline
\end{array}
\end{equation*}
\end{Thm}
\begin{proof}
%Define $a$ by
% $\Lambda^H/( \Lambda^G \oplus \Lambda^H_\iota )=\Z_2^{\oplus a}$.
We use $n$ and $a$ as in Lemma \ref{LEM_torsion_B_star}:
 $n=\rank \Lambda^H_\iota - a$ and $ \Lambda^H/( \Lambda^G \oplus \Lambda^H_\iota )\cong\Z_2^{\oplus a}$. 
Then
\begin{equation*}
 |\disc(\Lambda^G)| \cdot |\disc(\Lambda^H_\iota)|/|\disc(\Lambda^H)|
 =2^{2a}.
\end{equation*}
Recall that we determined $M_G=\Lambda^G$ and $N_G=\Lambda^H_\iota$ in Theorem \ref{G-inv lattice}.
Combined with Proposition \ref{symp-inv lattice}, $n$ is computed as in the following table.
%\begin{equation*}
% \mathrm{Tor}( (\Lambda^H)^\vee/ \{ x-\iota x \bigm| x\in (\Lambda^H)^\vee \} )
% \cong \Z_2^{\oplus n}, \quad
% n=\rank \Lambda^H_\iota - a,
%\end{equation*}
% where $n$ is computed as in the following table.
\begin{equation*} \label{EQ_list_n}
\begin{array}{|c|c|c|c|c|c|c|}
 \hline
 G & |\disc(\Lambda^H)| & |\disc(\Lambda^G)| & |\disc(\Lambda^H_\iota)|
  & a & \rank \Lambda^H_\iota & n \\
 \hline
 C_2                       &  1                & 2^{10}      & 2^{10}        & 10 & 12 & 2 \\ \hline
 C_2 \times C_2            &  2^8              & 2^8         & 2^{10}        &  5 &  8 & 3 \\ \hline
 C_2 \times C_2 \times C_2 & 2^{10}            & 2^6         & 2^8           &  2 &  6 & 4 \\ \hline
 D_6                       &  3^6              & 2^4 \cdot 3 & 2^4 \cdot 3^5 &  4 &  6 & 2 \\ \hline
 D_8                       &  2^{10}           & 2^4         & 2^{10}        &  2 &  5 & 3 \\ \hline
 D_{10}                    &  5^4              & 2^2         & 2^2 \cdot 5^4 &  2 &  4 & 2 \\ \hline
 D_{12}                    &  2^4 \cdot 3^4    & 2^2         & 2^4 \cdot 3^4 &  1 &  4 & 3 \\ \hline
 C_2 \times D_8            & 2^{10}            & 2^2         & 2^8           &  0 &  4 & 4 \\
 \hline
\end{array}
\end{equation*}
%Here the discriminant and rank of each lattice are given below.
Since $\mathrm{Br}(X) \cong \mathrm{Tor}(P)$, where $P$ is defined in (\ref{EQ_P_and_Q}), we have $\mathrm{Br}(X)\cong \Z_2^{\oplus n-1}$ by Lemma \ref{LEM_torsion_B_star}.
%By (\ref{EQ_Mayer-Vietoris}), (\ref{EQ_P_and_Q}) and Lemma \ref{LEM_torsion_B_star}, we have $\mathrm{Br}(X)\cong \Z_2^{\oplus n-1}$.
\end{proof}

We observe that $H_1(X,\Z)\cong \mathrm{Br}(X)\oplus \Z_2^2$ holds for Calabi--Yau threefolds of type K by the result of \cite{HK}. 
It is interesting to investigate this isomorphism via self-mirror symmetry of a Calabi--Yau threefold of type K. 

\begin{Cor} \label{cor: derived equiv}
Let $X_1$ and $X_2$ be Calabi--Yau threefolds of type K whose Galois groups are $G_1$ and $G_2$ respectively. 
Then a derived equivalence $\mathrm{D^bCoh}(X_1)\cong \mathrm{D^bCoh}(X_2)$ implies an isomorphism $G_1\cong G_2$. 
\end{Cor} 
\begin{proof}
The work \cite{Ad} of Addington shows that the finite abelian group $H_1(X,\Z)\oplus \mathrm{Br}(X)$ is a derived invariant of a Calabi--Yau threefold. 
It is also known that the Hodge numbers are derived invariants. 
These invariants completely determine the Galois group of a Calabi--Yau threefold of type K. 
We refers the reader to \cite{HK} for the Hodge numbers and  $H_1(X,\Z)$. 
\end{proof}

%%%%%%%%%%%%%%%%%%%%%%%%%%%%%%%%%%%%%%%%%%%%%%%%%%%%%%%%%%%%%%%%%%%%%%%%%%%%%%%%%%%%%%%%%%%%%%%%%%%%%%%%%%%%%

\subsection{Some Topological Properties}
%\subsection{type $L$}

In this section, we study some topological properties of Calabi--Yau threefolds of type K. 
We also refer the reader to \cite[Section 5]{HK}. 

\begin{Def}[\cite{KW}] \label{DEF_type_L}
Let $M$ be a free abelian group of finite rank and $\mu\colon M^{\otimes 3}\rightarrow \Z$ a symmetric trilinear form on $M$.  
Let $L$ be a lattice with bilinear form $\langle*,** \rangle_L \colon L\times L\rightarrow \Z$. 
We call $\mu$ of type $L$ if the following hold:
% (its scalar extension is defined similarly):
\begin{enumerate}
\item There is a decomposition $M\cong L\oplus N$ as an abelian group such that $N\cong \Z$. 
\item We have $\mu (\alpha,\beta,n)=n \langle\alpha,\beta \rangle_L$ ($\alpha,\beta \in L$, \ $n\in N\cong \Z$) 
and the remaining values are given by extending the form symmetrically and linearly and by setting other non-trivial values to be 0.
\end{enumerate}
\end{Def} 
%We refer the reader to \cite{KW} for general theory of trilinear forms of type $L$. 

\begin{Prop}\label{Tri}
%Let $X$ be a Calabi--Yau threefold of type K and $\pi:S\times E\rightarrow X$ be its minimal splitting covering with Galois group $G$. 
If $H$ is a cyclic group, the trilinear intersection form $\mu_X$ on $H^{2}(X,\Z)$ modulo torsion defined by the cup product is of type $L=M_G(1/2)$.
% for some lattice $L$. 
%When $G$ is isomorphic to $C_{2},D_{8},D_{10},D_{12}$ or $C_{2}\times D_{8}$, 
%The lattice $L$ is given by the following.
%\begin{center}
% \begin{tabular}{|c|c|c|c|c|c|}  \hline
% $G$  & $C_{2}$ &  $D_{8}$ &  $D_{10}$  &  $D_{12}$ & $C_{2}\times D_{8}$ \\ \hline
% $L$ & $U\oplus E_{8}$ &  $I_{1,2}$ &  $U$ & $U$ & $U$ or $I_{1,1}$  \\ \hline
% \end{tabular}
% \end{center}
\end{Prop}

\begin{proof}
%Assume first that $H$ is cyclic.
% $H\not \cong C_2^{\oplus 2}, C_2 \oplus C_4$. 
Let $0_E$ denote the origin of $E$ and pick a point $\mathrm{pt}_S$ in $S^H$, which is not empty \cite{Ni}. 
We define  
\begin{equation} \label{EQ_alpha_beta}
\alpha:=( \pi_*[S\times \{ 0_E \} ] )^{\mathrm{PD}} \in H^2(X,\Z), \ \ \ \beta:=(\pi_*[ \{ \mathrm{pt}_S \} \times E])^{\mathrm{PD}} \in H^4(X,\Z),
\end{equation}
%where we denote by $\pi_*$ the transfer map of $\pi$ and by $[C]^{\mathrm{PD}}$ the Poincar\'e dual of the class $[C]$ of a cycle $C$.
 where the superscript $\mathrm{PD}$ means the Poincar\'e dual.
Since $\pi$ is $2$-to-$1$ on $S \times \{ 0_E \}$, it follows that $\alpha$ is divisible by $2$ in $H^2(X,\Z)$.
Similarly, $\beta$ is divisible by $|H|$.
Note that $\frac{1}{2}\alpha\cup \frac{1}{|H|}\beta=1$.
We have natural isomorphisms
\begin{equation*}
 H^2(X,\Q) \cong H^2(S \times E)^G \cong
 (M_G \otimes \Q) \otimes H^0(E,\Q) \oplus \Q \alpha
 \cong (M_G \otimes \Q) \oplus \Q \alpha
\end{equation*}
 and
\begin{equation*}
 H^4(X,\Q) \cong H^4(S \times E)^G \cong
 (M_G \otimes \Q) \otimes H^2(E,\Q) \oplus \Q \beta
 \cong (M_G \otimes \Q) \oplus \Q \beta.
\end{equation*}

For a $\Z$-module $E$, we denote by $E_f$ its free part: $E_f:=E/\mathrm{Tor}(E)$.
By the exact sequence (\ref{EQ_Mayer-Vietoris}), it follows that $H_2(\overline{B^*},\Z)_f$ is embedded into $H_2(X,\Z)_f$ primitively.
By Proposition \ref{Nik Disc form}, we can identify $H_2(\overline{B^*},\Z)_f \cong H_2(S,\Z)/H_2(S,\Z)_G$ with $(M_G)^\vee$.
By the Poincar\'e duality, we have
$$
 H^2(X,\Z)_f  \cong
 \frac{1}{|G|} M_G \oplus \Z \frac{1}{2} \alpha, \ \ \ 
 H^4(X,\Z)_f  \cong
 (M_G)^\vee \oplus \Z \frac{1}{|H|} \beta.
$$
Hence the trilinear form $\mu_X$ is of type $L$ with $L=M_G(1/2)$ and $N=\Z \frac{1}{2} \alpha$.
In fact, we have 
\begin{equation*}
 \mu_X(\frac{1}{|G|}\gamma_1,\frac{1}{|G|}\gamma_2,n \cdot \frac{1}{2} \alpha)
 =|G|^2\cdot \frac{n}{2}
  \langle \frac{1}{|G|}\gamma_1,\frac{1}{|G|}\gamma_2\rangle_{M_G}
 =\frac{n}{2} \langle \gamma_1,\gamma_2 \rangle_{M_G}
\end{equation*}
 for $\gamma_1,\gamma_2\in M_G$ and $n \in \Z$.
%We will only prove the first assertion as we will not need the second in the rest of paper\footnote{
%For the second assertion, we need the classification of indefinite unimodular lattices. 
%For $n,m \ge 1$, there exists a unique odd unimodular lattice $I_{n,m}:=\langle 1 \rangle ^{\oplus n}\oplus \langle -1 \rangle ^{\oplus m}$ of signature $ (n,m)$,  up to isomorphism.
%}. 
%We will only prove the first assertion as we will not need the second in the rest of paper. 
%In this proof, we always work with cohomology groups modulo torsion. 
%Let $\pi:S\times E\rightarrow X$ be the minimal splitting covering. 
%
%We also set
%\begin{align}
%A:=[S\times 0_E]\in H_4(S\times E,\Z), \ \ \ B:=[\mathrm{pt}_S\times E] \in H_2(S\times E,\Z), \notag 
%\end{align}
%The degree of the map $\pi_*$ is computed as 
%$$
%\pi:A\stackrel{1:1}{\longrightarrow}A/H \stackrel{2:1}{\longrightarrow} A/G, \ \ \ 
%\pi:B\stackrel{|H|:1}{\longrightarrow}B/H \stackrel{1:1}{\longrightarrow} B/G.
%$$
%It thus follows that $\frac{1}{2}\alpha \in H^2(X,\Z)$, $\frac{1}{|H|}\beta \in H^4(X,\Z)$ and $\frac{1}{2}\alpha\cup \frac{1}{|H|}\beta=1$. 
%We then conclude that the trilinear form $\mu_X$ is of type $L$ for $N:= \frac{1}{|H|}\beta\Z$ and some lattice $L$ which contains $M_G$ of finite index. 
%When $H \cong C_2^{\oplus 2}$ (resp. $C_2 \oplus C_4$), we have $S^H=\emptyset$. 
%In this case, we take $\mathrm{pt}_S \in S^{C_2}$ (resp. $S^{C_4}$) and a similar argument works. 
%We leave the details to the reader. 
\end{proof}

\begin{Rem} \label{REM_non_cyclic}
For $H=C_2\times C_2$ and $C_2\times C_4$, the cohomology group $H^4(X,\Z)_f$ is generated by $(M_G)^\vee$ and $\frac{1}{|G|}\gamma$, where $\gamma$ is defined as follows.
We use the same notation as in Example \ref{EX_d8c2}.
Let $D$ denote the reducible curve in $\PP^1 \times \PP^1$ defined by $x^2 w^2+y^2 z^2=0$.
Then there is a $G$-equivariant continuous map $\varphi\colon E \rightarrow D$.
The inverse image in $S \times E$ of the graph of $\varphi$ is a $2$-cycle.
We define $\gamma$ to be the Poincar\'e dual of the image of this $2$-cycle in $X$, which is divisible by $|G|$.
\end{Rem}

\begin{Prop}
The second Chern class $c_2(X)$ is given by $c_2(X)=\frac{24}{|G|}\beta$, where $\beta$ is defined in (\ref{EQ_alpha_beta}).  
\end{Prop}
\begin{proof}
Let $\beta_S$ be a generator of $H^4(S,\Z)$ with $\int_S\beta_S=1$. 
By the universal property of characteristic classes, we have $\pi^*c_2(X)=c_2(S\times E)=24p_1^*\beta_S$, 
where $p_1 \colon S\times E\rightarrow S$ is the first projection.   
Applying the transfer map $\pi_*$, we obtain $|G|c_2(X)=24\pi_*p_1^*\beta_S=24\beta$. 
\end{proof}

%Together with the results obtained in \cite{HK}, we have determined the topological properties of Calabi--Yau threefolds of type K. 

%%%%%%%%%%%%%%%%%%%%%%%%%%%%%%%%%%%%%%%%%%%%%%%%%%%%%%%%%%%%%%%%%%%%%%%%%%%%%%%%%%%%%%%%%%%%%%%%%%%%%%%%%%%%%
%%%%%%%%%%%%%%%%%%%%%%%%%%%%%%%%%%%%%%%%%%%%%%%%%%%%%%%%%%%%%%%%%%%%%%%%%%%%%%%%%%%%%%%%%%%%%%%%%%%%%%%%%%%%%

%Since there is no precise definition of mirror symmetry which can be appealed to, various types of evidence are collected, 
%including an examination of the Yukawa couplings in the limit and construction of SYZ fibrations.

\section{Yukawa Couplings} \label{Mirror Symmetry}
We begin our discussion on mirror symmetry. 
Calabi--Yau threefolds $X$ of type K are, by the construction, a topological self-mirror threefolds, that is, $h^{1,1}(X)=h^{1,2}(X)$.    
However, mirror symmetry should involve more than the mere exchange of Hodge numbers. 
%and a brief investigation reveals that mirror symmetry of Calabi--Yau threefolds of type K is more interesting than the reader might expect. 
We will see that the mirror symmetry of Calabi--Yau threefolds of type K bears a resemblance to the mirror symmetry of Borcea--Voisin threefolds \cite{Vo, Bo}.  
The latter relies on the strange duality of certain involutions of K3 surfaces discovered by Nikulin \cite{Ni3}. 
Our study, on the other hand, will employ the classification of lattices $M_G$ and $N_G$ (Theorem \ref{G-inv lattice}). 
%To the reader looking for a richer exposition, we suggest the sources \cite{Vo, Do, CK}. \\

As a warm-up, let us consider the case $G\cong C_2$. 
A generic K3 surface $S$ with an Enriques involution is a self-mirror K3 surface in the sense of Dolgachev \cite{Do}, that is,
\begin{equation}
U \oplus NS(S) \cong T(S) \label{Enriques duality}
\end{equation}
where $NS(S)$ and $T(S)$ are given as $M_G$ and $N_G$ respectively in this case. 
Thus it is no wonder that the corresponding Enriques Calabi--Yau threefold (Example \ref{Enr CY3}) is a self-mirror threefold. 
%This has been studied to a great extent in the context of mirror symmetry and string theory \cite{KM,MP}. 
%In the case $G\cong C_2^{\oplus 2}$, we still have 
%\begin{equation}
%U \oplus M_G \cong N_G. \label{duality}
%\end{equation}
For $G \not \cong C_2$, the corresponding K3 surface cannot be self-mirror symmetric as $\rk T(S)<12$. 
Nevertheless, as we will see, the Calabi--Yau threefold $X=(E\times S)/G$ is self-mirror symmetric. 
In general, $M_G$ and $N_G$ do not manifest symmetry over $\Z$, 
but the duality
\begin{equation}
U \oplus M_G \cong N_G \label{duality}
\end{equation}
still holds over $\Q$ or an extension of $\Z$ (Remark \ref{REM_Q_isom}).  
Note that $M_G$ is not equal to $NS(S)$ but to $NS(S)^G$, while $T(S)=N_G$ generically holds. 
The duality (\ref{duality}) can be thought of as an $H$-equivariant version of the duality (\ref{Enriques duality}) 
since we have $M_G=(H^2(S,\Z)^H)^\iota$ and $N_G=(H^2(S,\Z)^H)_\iota$.

\subsection{Moduli Spaces and Mirror Maps} \label{Moduli Spaces and Mirror Maps}
Throughout this section, $X$ is a Calabi--Yau threefold of type K and $S\times E\rightarrow X$ is its minimal splitting covering with Galois group $G$. 
We will analyze the moduli space of $X$ equipped with a complexified K\"{a}hler class. 
%Leaving some discrete identifications, 
Such a moduli space locally splits into the product of the complex moduli space and the complexified K\"{a}hler moduli space. 
We will also sketch the mirror map.% It also saves to set conventions for later sections. 
%Although there are extra discrete identifications on the full moduli spaces, we will ignore them below. \\
%Hence, in this section, we assume that the moduli space is the product of the complex moduli space and the complexified K\"{a}hler moduli space. \\

We begin with the moduli space of marked K3 surfaces $S$ with a Calabi--Yau $G$-action (Definition \ref{def_cyg_k3}).   
%Let $\psi:G\rightarrow \mathrm{O}(\Lambda)$ be the induced $G$-action on $\Lambda \cong H^2(S,\Z)$. 
%Set  $M_G=H^{2}(S,\Z)^{G}$ and $N_G=H^{2}(S,\C)_{C_{2}}^{H}$. 
The period domain of such K3 surfaces is given by 
$$
\mathcal{D}_{S}^{G}:=\{\C\omega_{S} \in \mathbb{P}(N_G \otimes \C) \bigm| \langle \omega_{S},\omega_{S}\rangle=0, \  \langle\omega_{S},\overline{\omega}_{S}\rangle>0\}. 
$$
The extended $G$-invariant complexified K\"{a}hler moduli space\footnote{For simplicity, we use this definition, even though $\kappa_S$ may not be a K\"{a}hler class.} is defined as the tube domain 
$$
\mathcal{K}_{S,\C}^{G}:=\{B_S+i\kappa_{S}\in M_G\otimes \C \bigm| \kappa_{S}\in \mathcal{K} \},
$$
 where $\mathcal{K}$ is the connected component of $\{ \kappa \in M_G \otimes \R \bigm| \langle \kappa,\kappa \rangle > 0 \}$ containing a K\"{a}hler class.
%Both $\mathcal{D}_{K3}^{G,\psi}$ and $\mathcal{K}_{K3}^{G,\psi}$ are symmetric domains of type IV. 
%The class $B_S$, which is usually defined modulo $H^2(S,\Z)$, is called the $B$-field%
The class $B_S$ is called the $B$-field, and usually defined modulo $H^2(S,\Z)$. %\footnote{
%We refer the reader to \cite[Conjecture 6.6]{Gr2} for an interpretation of the $B$-field. 
%It is related to the Leray spectral sequence as discussed in Section \ref{Mirror Map via Leray Spectral Sequence}.}.  
This is where the Brauer group comes into the play, 
but for a moment we regard $\mathcal{K}_{S,\C}^{G}$ as a covering of the $G$-invariant complexified K\"{a}hler moduli space $\mathcal{K}_{S,\C}^{G}/H^2(S,\Z)$,  
analogous to $\mathcal{D}_{S}^{G}$ being a covering of the complex moduli space. 
We construct a mirror map at the level of theses coverings. 

\begin{Prop} \label{Mirror Map}
There exists a holomorphic isomorphism $q^{G}_{S} \colon \mathcal{K}_{S,\C}^{G} \rightarrow \mathcal{D}_{S}^{G}$.    
\end{Prop}
\begin{proof}
%By Theorem \ref{G-inv lattice} %and Witt's cancellation theorem, 
There exists an isomorphism $N_G\otimes\Q\cong (U\oplus M_G) \otimes \Q$ as quadratic spaces over $\Q$ (Remark \ref{REM_Q_isom}). 
The holomorphic isomorphism $q^{G}_{S} \colon \mathcal{K}_{S,\C}^{G}\rightarrow \mathcal{D}_{S}^{G}$ is then given by the mapping 
$$
B_S+i\kappa_{S} \mapsto \C(e-\frac{1}{2}\langle B_S+i\kappa_{S}, B_S+i\kappa_{S} \rangle f+B_S+i\kappa_{S}), 
$$
where $e$ and $f$ are the standard basis of $U$. % i.e. $e^2=f^2=0$, $\langle e,f \rangle =1$. 
This is known as the tube domain realization \cite{Do, Vo}. 
\end{proof}

Let us next recall mirror symmetry of the elliptic curves. 
Consider an elliptic curve $E=\C/(\Z\oplus \Z\tau_{E})$ with a Calabi--Yau $G$-action, 
by which we mean a $G$-action on $E$ described in Proposition \ref{PROP_OS}. 
We can continuously deform the $G$-action as we vary the moduli parameter.    
The period domain of such elliptic curves is thus given by 
$$
\mathcal{D}_{E}:=\{\C\phi \in \PP(H^1(E,\Z)\otimes \C) \bigm| i\int_E\phi \cup \overline{\phi} >0\}. 
$$
%for $\langle\phi, \overline{\phi}\rangle:=i\int_E\phi \cup \overline{\phi}$. 
Given a symplectic basis $(\alpha,\beta)$ of $H^1(E,\Z)$, we may normalize $\phi=\alpha +\tau_E \beta$ for $\tau_E \in \HH$. 
%Henceforth we use the identification $\mathcal{D}_{E} \cong \HH$. 
The extended complexified K\"{a}hler moduli space of $E$ is identified, through integration over $E$, with the set 
$$
\mathcal{K}_{E,\C}:=\{B_E+i\kappa_{E} \in H^{1,1}(E)\cong \C \bigm| \kappa_{E}>0\}. 
$$
There is an isomorphism $q_{E} \colon \mathcal{K}_{E,\C}\rightarrow \mathcal{D}_{E}$ given by 
\begin{equation}
q_{E}(B_E+i\kappa_{E}):=\C(\alpha+(B_E+i\kappa_{E})\beta). \label{Mirror Map Elliptic}
\end{equation} 
We associate to a pair $(E_1,B_{E_1}+i\kappa_{E_1})$ the mirror pair $(E_2,\tau_{E_1})$ with $\tau_{E_2}:=B_{E_1}+i\kappa_{E_1}$. 

Given a point $(B_S+i\kappa_{S},B_E+i\kappa_{E})$ in the extended complexified K\"{a}hler moduli space $\mathcal{K}_{X,\C}:=\mathcal{K}_{S,\C}^{G}\times \mathcal{K}_{E,\C}$ of $X$, 
the point 
$$
(q^{G}_{S}(B_S+i\kappa_{S}), q_E(B_E+i\kappa_{E}))\in \mathcal{D}_{X}:=\mathcal{D}_{S}^{G}\times \mathcal{D}_{E}
$$  
determines a pair $(S^\vee,E^\vee)$ of a K3 surface and an elliptic curve, 
and hence another  Calabi--Yau threefold $X^\vee:=(S^\vee\times E^\vee)/G$ of type K\footnote{Precisely, the set of points in $\mathcal{D}_X$ corresponding to Calabi--Yau threefolds of type K is an open dense subset of $\mathcal{D}_X$ (see \cite{HK}).}.
In the same manner, a point $(\omega_{S},\tau_{E})\in \mathcal{D}_{X}$ determines 
an extended complexified K\"{a}hler structure $((q_{S}^{G})^{-1}(\omega_{S}),q_E^{-1}(\tau_E))\in \mathcal{K}_{X,\C}$ on the pair $(S^\vee, E^\vee)$, 
which subsequently determines an extended complexified K\"{a}hler structure on $X^\vee$. 
Therefore we can think of the map $q=(q_{S}^{G},q_{E})\colon \mathcal{K}_{X,\C} \rightarrow \mathcal{D}_X$ %of $\mathcal{D}_{X}\times \mathcal{K}_{X,\C}$ 
 as the {\it self-mirror map} of the family of Calabi--Yau threefolds $X$.  
%We will come back to this point in Section \ref{Mirror Map via Leray Spectral Sequence}. 

%%%%%%%%%%%%%%%%%%%%%%%%%%%%%%%%%%%%%%%%%%%%%%%%%%%%%%%%%%%%%%%%%%%%%%%%%%%%%%%%%%%%%%%%%%%%%%%%%%%%%%%%%%%%%

%%%%%%%%%%%%%%%%%%%%%%%%%%%%%%%%%%%%%%%%%%%%%%%%%%%%%%%%%%%%%%%%%%%%%%%%%%%%%%%%%%%%%%%%%%%%%%%%%%%%%%%%%%%%%

\subsection{A- and B-Yukawa Couplings}
We now introduce the A-Yukawa coupling with a complexified K\"{a}hler class $B+i\kappa \in \mathcal{K}_{X,\C}/H^2(X,\Z)$. 
The A-Yukawa coupling $Y_{A}^X \colon H^{1,1}(X)^{\otimes3}\rightarrow \C [[{\bold q}]]$ is the symmetric trilinear form given by
$$
Y_{A}^X(H_{1},H_{2},H_{3})
=H_{1}\cup H_{2}\cup H_{3}
+\sum_{\beta\ne0}n_{\beta}\frac{e^{2\pi i \int_{\beta}(B+i\kappa)}}{1-e^{2\pi i \int_{\beta}(B+i\kappa)}}\prod_{i=1}^{3}\int_{\beta}H_{i}, 
$$ 
where the sum is over the effective classes $\beta\in H_{2}(X,\Z)$ and $\C [[{\bold q}]]$ is the Novikov ring. 
The number $n_{\beta}$ is naively the number of rational curves on $X$ in the homology class $\beta$. 
A mathematical definition of $n_{\beta}$ is carried out via the Gromov--Witten invariants of $X$.    
%Note that the first term is simply the cup product and the second term is the quantum correction term of the classical coupling.  
%Suppose that the complexified K\"{a}hler class $B+i\kappa \in \mathcal{K}_{X,\C}/H^2(X,\Z)$ approaches infinity, 
%which we call the {\it large volume limit} (LVL) of $\mathcal{K}_{X,\C}/H^2(X,\Z)$, in the sense that 
%$\int_{\beta}\kappa\rightarrow \infty$ for all effective curves class $\beta\in H_{2}(X,\Z)$. 
At the {\it large volume limit} (LVL) of $H^2(X,\R)/H^2(X,\Z)$, where $\int_{\beta}\kappa\rightarrow \infty$ for all effective curve classes $\beta\in H_{2}(X,\Z)$, 
the A-Yukawa coupling $Y_A$ asymptotically converges to the cup product as the quantum correction terms vanish\footnote{For Calabi--Yau threefolds of type K, no quantum correction in $Y_{A}^X$ is expected because the moduli space of $g=0$ stable maps probably involves a factor of $E$. This is also examined in Section \ref{B-Model Computation}.}. 

It is instructive to briefly review a role of the Brauer group in mirror symmetry \cite{AM}. 
In the A-model topological string theory, the correlation functions depend on {\it instantons}, 
namely holomorphic maps $\sigma \colon \Sigma_g\rightarrow X$ from the world-sheet $\Sigma_g$ to the target threefold $X$. 
The action $\mu$ is required to depend linearly on the homology class of the image of the map; $\mu\in \Hom(H_2(X,\Z),\C^\times) \cong H^2(X,\C^\times)$. 
%An important observation is that this space is larger than the conventional complexified K\"ahler moduli space $\mathcal{K}_{X,\C}$. 
Note that $\mathrm{Tor}(H^2(X,\C^\times))$ is isomorphic to $\mathrm{Br}(X)$ as an abstract group.
By the exact sequence
$$
0\rightarrow H^2(X,\Z)/\mathrm{Tor}(H^2(X,\Z)) \rightarrow H^2(X,\C) \rightarrow H^2(X,\C^\times) \rightarrow \mathrm{Br}(X) \rightarrow 0
$$
%and the isomorphism $\mathrm{Tor}(H_2(X,\Z))\cong \mathrm{Tor}(H^3(X,\Z))$, 
we see that 
$$
H^2(X,\C^\times) \cong \left(H^2(X,\C)/H^2(X,\Z)\right) \oplus \mathrm{Br}(X). 
$$
%provided that the sequence splits. 
Assuming that the B-field represents a class in $H^2(X,\R)/H^2(X,\Z)$, 
we think of the complexified K\"ahler moduli space $\mathcal{K}_{X,\C}$ lying in the first summand. 
Then the above isomorphism shows that we need to incorporate in the A-model moduli space the contribution coming from $\mathrm{Br}(X)$.  
%In other words, a discrete parameter $\alpha \in \mathrm{Br}(X)$ is required in order to determine the A-model correlation functions. 
However, the choice of $\alpha \in \mathrm{Br}(X)$ does not matter at the LVL, and the components of the moduli space\footnote{No flop is allowed \cite[Proposition 5.5]{HK}.},  
each parametrized by $B+i\kappa \in \mathcal{K}_{X,\C}/H^2(X,\Z)$ 
but having different values of $\alpha$, are joined at the LVL (Figure \ref{fig:A-model moduli}). 
\begin{figure}[htbp]
 \begin{center} 
  \includegraphics[width=30mm]{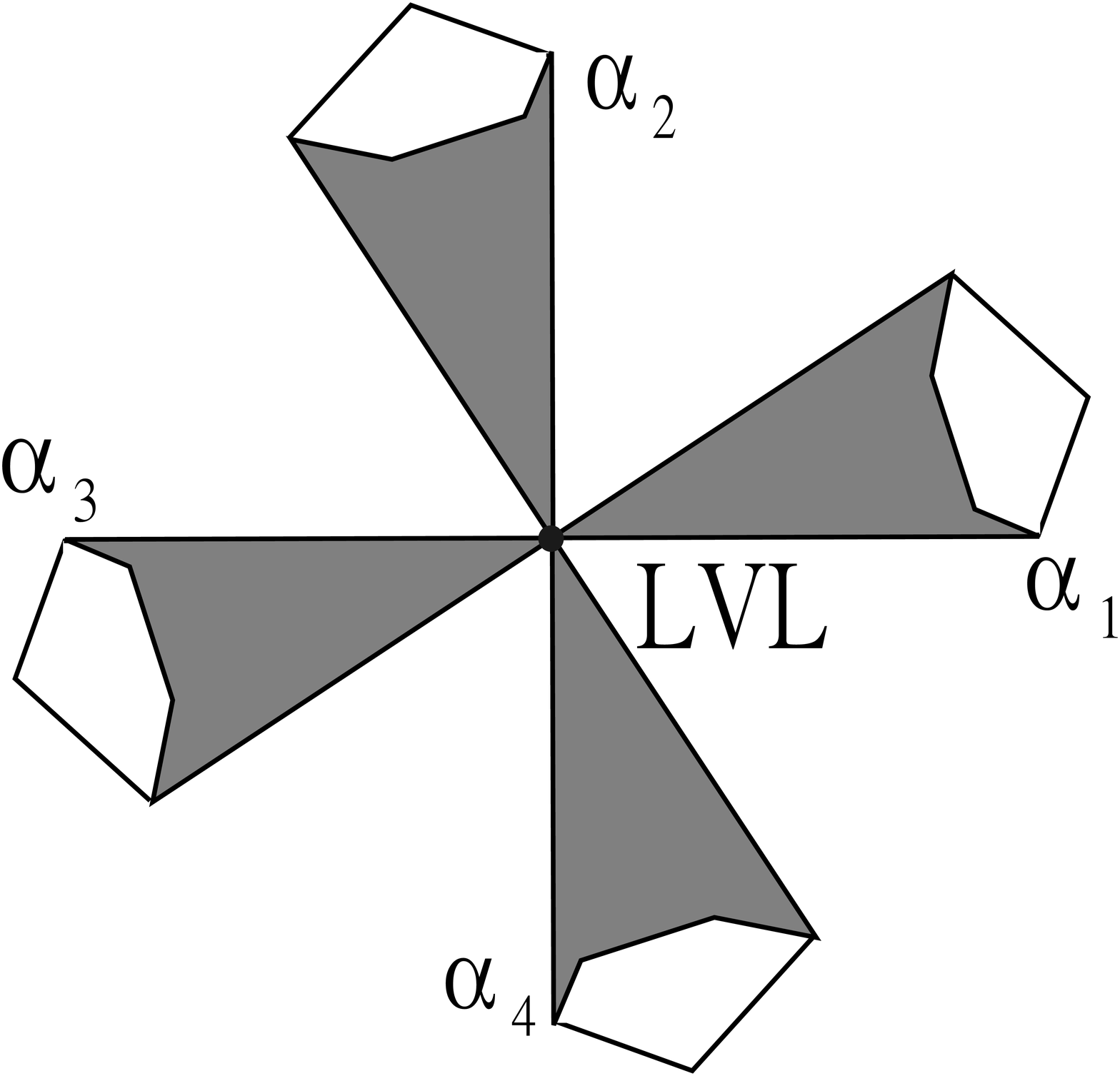}
 \end{center}
 \caption{A-model moduli space for $\mathrm{Br}(X)\cong \Z_2^{\oplus2}$}
 \label{fig:A-model moduli}
\end{figure}
For example, it is interesting to study the role of Brauer group in the space of Bridgeland stability conditions. 
For the rest of this paper, we will ignore the issue of the Brauer group without harm, by assuming that such an $\alpha$ is chosen. 

We next define the B-Yukawa coupling.   
First recall the isomorphism $H^{2,1}(X)\cong H^{1}(X,T_{X})$. 
%the first isomorphism is due to the Dolbeaut's theorem and the second one is due to the fact $\Omega_{X}^{2} \cong \Theta_{X}$ for a Calabi--Yau threefold $X$. 
The differential of the period map on the Kuranishi space for $X$ is identified with the interior product $H^{1}(X,T_{X})\rightarrow \Hom(H^{k,l}(X),H^{k-1,l+1}(X))$ 
for $k,l \in \N$. 
 By iterating these, we obtain a symmetric trilinear form, called the B-Yukawa coupling, 
 $$
 Y_{B}^X \colon H^{1}(X,T_{X})^{\otimes3}\longrightarrow \Hom(H^{3,0}(X),H^{0,3}(X))\cong \C, 
 $$
 where the last identification depends on a trivialization $H^{3,0}(X) \cong \C$. 
% For Calabi--Yau threefolds of type K, this ambiguity does not matter much, as we will see below. 
%Let $X$ be a Calabi--Yau threefold of type K with the minimal splitting covering $S\times E\rightarrow X$ and Galois group $G$.   
Recall that we have the natural isomorphism
\begin{equation*}
  H^{2,1}(X)  \cong H^{1,1}(S)^H_{C_2} \otimes H^{1,0}(E)\oplus H^{2,0}(S) \otimes H^{0,1}(E).
\end{equation*}
Choose holomorphic volume forms $dz_E$ and $\omega_S$ on $E$ and $S$ respectively. 
Using the isomorphism $\varphi \colon T_{S}\rightarrow \varOmega_{S}, v \mapsto \iota_v\omega_S=\omega_S(v,*)$, 
we identify the following:
\begin{align}
H^{1,1}(S)_{C_{2}}^{H}\otimes H^{1,0}(E) & \cong H^{1}(S,T_{S}),\ \ \ \eta\wedge \overline{\theta} \otimes dz_{E} \leftrightarrow \varphi^{-1}(\eta)\otimes\overline{\theta}; \notag  \\
H^{2,0}(S) \otimes H^{0,1}(E) & \cong H^{1}(E,T_{E}), \ \ \ \omega_{S}\otimes d\overline{z}_{E}\leftrightarrow \frac{\partial}{\partial z_{E}}\otimes d\overline{z}_{E}. \notag
\end{align}
Here $\eta \ (\overline{\theta})$ is the (anti-)holomorphic part of $\eta\wedge \overline{\theta}\in H^{1,1}(S)_{C_{2}}^{H}$.    
\begin{Lem}
Assume that the trivialization of $H^{3,0}(X)\cong H^{2,0}(S)\otimes H^{1,0}(E)$ is given by a nowhere-vanishing global section $\omega_{S}\otimes dz_{E}$. 
Then we have 
$$
Y_{B}^X(\omega_{S}\otimes d\overline{z}_{E}, \eta_{1}\wedge \overline{\theta}_{1} \otimes dz_{E},\eta_{2}\wedge \overline{\theta}_{2} \otimes dz_{E})
=\langle\eta_{1}\wedge \overline{\theta}_{1},\eta_{2}\wedge \overline{\theta}_{2} \rangle_{H^{1,1}(S)_{C_{2}}^{H}}
$$
where $\langle*,**\rangle_{H^{1,1}(S)_{C_{2}}^{H}}$ denotes the cup product restricted to $H^{1,1}(S)_{C_{2}}^{H}$. 
The B-Yukawa coupling $Y_{B}^X$ is obtained by extending the above form trilinearly and symmetrically and by setting other non-trivial couplings to be 0. 
\end{Lem}
\begin{proof}
Using the above identification of cohomology groups, we may show the assertion by a straightforward calculation.   
\end{proof}
Therefore the B-Yukawa coupling is of type $\langle*,**\rangle_{H^{1,1}(S)_{C_{2}}^{H}}$ after $\C$-extension (see Definition \ref{DEF_type_L}). 
%and the residual quadratic form may be identified with the quadratic form on $H^{1,1}(S)_{C_{2}}^{H}$. 
On the other hand, by Theorem \ref{G-inv lattice}, there always exists a decomposition  
$$
N_G=U(k) \oplus  U(k)^{\perp}_{N_G}, \ \ \  U(k)^{\perp}_{N_G}\otimes \Q \cong M_G\otimes \Q,
$$
for some $k \in \N$. 
%Moreover, $U(k)^{\perp}\cong M_G$ holds for $G\cong C_2^{\oplus n}\ (1\le n \le 3)$. 
The complex moduli space of $S$ admits the Bailey--Borel compactfication 
and the primitive isotropic sublattices of $T(S)=N_G$ correspond to the cusps of the compactification \cite{Sc}. 
Therefore,
% once we choose the standard basis $e,f$ of the hyperbolic lattice $U(k)$, 
 an isotropic vector $e \in U(k)$ corresponds to a {\it large complex structure limit} (LCSL) in the context of mirror symmetry. 
%We will not delve into the details of LCSL, but suggest the exposition \cite{CK} to the reader. \\

Let us now investigate the behavior of the B-Yukawa coupling $Y_B^X$ when the moduli point of our K3 surface $S$ approaches to the cusp. 
%to put it another way, when the period point $\C\omega_{S}$ approaches to the line $\C e$. 
The Hodge $(1,1)$-part of the quadratic space $H^{2}(S,\C)_{C_{2}}^{H}$ is identified with the quotient quadratic space
$$
H^{1,1}(S)_{C_{2}}^{H}\cong(\omega_{S})_{H^{2}(S,\C)_{C_{2}}^{H}}^{\perp}/\C\omega_{S}. 
$$
Hence, at the limit as the period $\omega_{S}$ approaches to $e$, 
the quadratic space $H^{1,1}(S)_{C_{2}}^{H}$ becomes the $\C$-extension of $e^\perp/\Q e\cong M_{G}\otimes \Q$. 
%$$
%\lim_{\omega_{S}\to e}\langle \omega_S,x \rangle=0 \Longleftrightarrow x \in (\Z e\oplus H^{2}(S,\Z)^{G})\otimes \C
%$$
Thereby, the quadratic factor of $Y_B^X$ asymptotically converges to $M_G \otimes \C$ near the LCSL. 
%as the period point $\omega_S$ approaches to the boundary point $e$ of the period domain $\mathcal{D}^G_S$.
%This observation is compatible with the asymptotic behavior of the A-Yukawa coupling as discussed later.    

Given a pair $(\mathscr{X},\mathscr{Y})$ of mirror families of Calabi--Yau threefolds. 
One feature of mirror symmetry is the identification of the Yukawa couplings 
$$
Y_{A}^{X}(H_{1},H_{2},H_{3})=Y_{B}^{Y}(\theta_{1},\theta_{2},\theta_{3})
$$
after a transformation, called the {\it mirror map}, of local moduli parameters $H_i$
around the LVL $\mathcal{K}_{X,\C}/H^2(X,\Z)$ and $\theta_i$ around a LCSL of $\mathcal{M}_Y$. 
Here $X$ is a general member of $\mathscr{X}$ near the LVL and $Y$ is a general member of $\mathscr{Y}$ near the LCSL. 
%This highly non-trivial conjecture has been confirmed for a large class of Calabi--Yau threefolds, for example, the complete intersections of hypersurfaces in toric varieties \cite{Gi,LLY}. 
%Our claim here is that Calabi--Yau threefolds of type K are self-mirror threefolds in the sense that $Y_{A}^X$ and $Y_{B}^X$ have similar asymptotic behavior.  

 \begin{Thm} \label{Yukawa}
 Let $X$ be a Calabi--Yau threefold of type K. 
 The asymptotic behavior of the A-Yukawa coupling $Y_A^X$ around the LVL coincides, up to scalar multiplication,
 with that of the B-Yukawa coupling $Y_B^X$ around a LCSL. 
 The identification respects the rational structure of the trilinear forms. 
 \end{Thm}
 \begin{proof}
The assertion readily follows from the discussions in this section. 
As the complexified K\"{a}hler moduli $B+i\kappa \in \mathcal{K}_{X,\C}/H^2(X,\Z)$ approaches to the LVL, 
the A-Yukawa coupling $Y_A^X$ becomes the classical cup product on $H^{1,1}(X)$, 
which is the $\C$-extension of the trilinear form on $H^2(X,\Z)$.
By Theorem \ref{Tri} and Remark \ref{REM_non_cyclic}, the trilinear form on $H^2(X,\Z)$ is of type $L=M_G(1/2)$ over $\Q$.
On the other hand, as the period $\omega_{S}$ of the K3 surface $S$ approaches to the cusp of the period domain $\mathcal{D}^{G}_{S}$, 
the B-Yukawa coupling $Y_B^X$ becomes a trilinear form, which has a linear factor and whose residual quadratic form is the $\C$-extension of $e^\perp/\Q e\cong M_{G}\otimes \Q$. 
 \end{proof}

The identification in the proof of Theorem \ref{Yukawa} respects the integral structure, that is, $e^\perp/\Z e \cong M_{G}$, for some cases including $G\cong C_2$.
However, this is not true in general (see Theorem \ref{G-inv lattice}).
% the free part $H^2(X,\Z)/\mathrm{Tor}$ is in general contained in $H^2(S\times E,\Z)^G$ of finite index via $\pi^*$ 
%The difficulty in putting an integral structure is apparent also in Section \ref{Mirror Map via Leray Spectral Sequence}. 
%See also \cite{Vo} for the case of Borcea--Voisin threefolds. 

%%%%%%%%%%%%%%%%%%%%%%%%%%%%%%%%%%%%%%%%%%%%%%%%%%%%%%%%%%%%%%%%%%%%%%%%%%%%%%%%%%%%%%%%%%%%%%%%%%%%%%%%%%%%%

\subsection{B-Model Computation} \label{B-Model Computation}
In this section, we will compute the B-models of Calabi--Yau threefolds of type K with 3- or 4-dimensional moduli spaces. 
%Let $X$ be a Calabi--Yau threefold of type K and $\pi:S\times E\rightarrow X$ its minimal splitting covering with Galois group $G$. 
The necessary integrals reduce to those of $S$ and $E$, and thus there should be no instanton corrections. 
However it is interesting to see how they recover the lattice $M_G$. %\footnote{Note that this is not about mirror symmetry of K3 surfaces.}.  
We will also find some modular properties in the computation, which may be useful in further investigations, such as the BCOV theory.   
For $G\cong D_{12}$, with the same notation as in Example \ref{Ex $D_{10}$}, the period integral of $S$ is,  
for the cycle $\gamma$ given by $|\frac{x}{y}|=|\frac{z}{w}|=\epsilon \ (0 < \epsilon \ll 1)$,  
$$
\Phi_0:=\int_{\gamma} \frac{(xdy-ydx)\wedge (zdw-wdz)}{\sqrt{Ax^2y^2z^2w^2+B(x^4z^4+y^4w^4)+C(x^4zw^3+y^4z^3w)}},
$$
where $[A:B:C]\in \PP^2$ is the moduli parameter.  
On the open set $\{A\ne 0\} \cong \C^2$ with $z_1:=(B/2A)^2$ and $z_2:=(C/2A)^2$, the period integral reads
$$
\Phi_0(z_1,z_2)=1+12(z_1+z_2)+420(z_1^2+4z_1z_2+z_2^2)+18480(z_1+z_2)(z_1^2+8z_1z_2+z_2^2)+\cdots. 
$$
%for an indivisible integral cycle $\gamma$ invariant under monodromy about $(z_1,z_2)=(0,0)$. 
A numerical computation shows that the Picard--Fuchs system is generated by 
$$
\Theta_i^2-4z_i(4\Theta_1+4\Theta_2+3)(4\Theta_1+4\Theta_2+1)
$$
for $i=1,2$, where $\Theta_i:=z_i\frac{\partial}{\partial z_i}$ is the Euler differential. 
%This is the same as the one previously obtained in a slightly different context in \cite{KM}. %\footnote{It seems that \cite{KM} contains some typos.}. 
We observe that the point $(z_1,z_2)=(0,0)$ is a LCSL of this family. 
The linear-logarithmic solutions are
\begin{align}
\Phi_1(z_1,z_2):=\Phi_0\log(z_1)+40z_1+64z_2+1556z_1^2+7904z_1z_2+2816z_2^2+\cdots, \notag \\
\Phi_2(z_1,z_2):=\Phi_0\log(z_2)+64z_1+40z_2+2816z_1^2+7904z_1z_2+1556z_2^2+\cdots. \notag
\end{align}
Then the mirror map around the LCSL is given by
\begin{align}
q_1=\exp(\Phi_1/\Phi_0)=z_1+8z_1(5z_1+8z_2)+4z_1(469z_1^2+2304z_1z_2+1024z_2^2)+\cdots, \notag \\
q_2=\exp(\Phi_2/\Phi_0)=z_2+8z_2(8z_1+5z_2)+4z_2(1024z_1^2+2304z_1z_2+469z_2^2)+\cdots. \notag
\end{align}
%\begin{align}
%z_1=q_1-40q_1^2-64q_1q_2+1324q_1^3+2560q_1^2q_2+2560q_1q_2^2-39872q_1^4+\dots, \notag \\
%z_2=q_2-64q_1q_2-40q_2^2+2560q_1^2q_2+2560q_1q_2^2+1324q_2^3-84736q_1^3q_2+\dots, \notag
%\end{align}
We write $q_i=e^{2\pi i t_i} \ (i=1,2)$ and we regard $(t_1,t_2)$ as the $G$-invariant complexified K\"ahler parameter of $S$ 
so that it descends to the complexified K\"ahler parameter of $X$. 
The techniques in \cite{LY,KM} shows that the inverse mirror maps read
\begin{align}
z_1(q_1,q_2)&=\frac{\vartheta_2^8(q_1)}{64(\vartheta_3^4(q_1)+\vartheta_4^4(q_1))^2}\left(1-\frac{\vartheta_2^8(q_2)}{(\vartheta_3^4(q_2)+\vartheta_4^4(q_2))^2}\right),\notag \\
z_2(q_1,q_2)&=\frac{\vartheta_2^8(q_2)}{64(\vartheta_3^4(q_2)+\vartheta_4^4(q_2))^2}\left(1-\frac{\vartheta_2^8(q_1)}{(\vartheta_3^4(q_1)+\vartheta_4^4(q_1))^2}\right), \notag 
\end{align}
where $\vartheta_2(q),\vartheta_3(q)$ and $\vartheta_4(q)$ are the Jacobi theta functions\footnote{
$
\vartheta_2(q):=\sum_{n \in \Z}q^{\frac{1}{2}(n+\frac{1}{2})^2}, %= 2 q^{1/8}\prod_{n=1}^\infty(1-q^n)(1+q^n)^2, 
\vartheta_3(q):=\sum_{n \in \Z}q^{\frac{n^2}{2}}, %=\prod_{n=1}^\infty(1-q^n)(1+q^{n-\frac{1}{2}})^2 
\vartheta_4(q):=\sum_{n \in \Z}(-1)^nq^{\frac{n^2}{2}}. %=\prod_{n=1}^\infty(1-q^n)(1-q^{n-\frac{1}{2}})^2.
$
}. 
The period $\Phi_0$ also reads, with respect to the mirror coordinates, 
$$
\Phi_0(q_1,q_2)=\frac{1}{2}\sqrt{(\vartheta_3^4(q_1)+\vartheta_4^4(q_1))(\vartheta_3^4(q_2)+\vartheta_4^4(q_2))}.
$$
We denote by $\nabla_{z_i}$ the Gauss--Manin connection with respect to the local moduli parameter $z_i$ and define 
$C_{z_i,\dots,z_j}:=\int_{S}\omega_S\wedge (\nabla_{z_i}\cdots\nabla_{z_j}\omega_S)$. 
Using the Griffith transversality relations 
$$
C_{z_i}=0, \ \ \ 3\partial_{z_1}C_{z_1,z_1}=2C_{z_1,z_1,z_1}, \ \ \ \partial_{z_2}C_{z_1,z_1}+2\partial_{z_1}C_{z_1,z_2}=2 C_{z_1,z_1,z_2}, 
$$
we determine the B-Yukawa couplings, up to multiplication by a constant, as follows:
\begin{align}
C_{z_i,z_i}(z_1,z_2)&=\frac{1}{2^{5}z_i(1-2^7(z_1+z_2+2^6z_1z_2)+2^{12}(z_1^2+z_2^2))} \ \ (i=1,2),\notag \\
C_{z_1,z_2}(z_1,z_2)&=\frac{1-64(z_1+z_2)}{2^{12}z_1z_2(1-2^7(z_1+z_2+2^6z_1z_2)+2^{12}(z_1^2+z_2^2))}. \notag
\end{align}
Via the mirror map, we determine the A-Yukawa couplings 
$$
K_{t_i,t_j}(q_1,q_2)=\frac{1}{\Phi_0(z_1,z_2)^2}\sum_{k,l=1}^2C_{z_k,z_l}(z_1,z_2)\frac{\partial z_k}{\partial t_i}\frac{\partial z_l}{\partial t_j}. 
$$
There is no quantum correction for a K3 surface and we confirm that $K_{t_1,t_1}=K_{t_2,t_2}=0$ and $K_{t_1,t_2}=1$. 
This recovers the fact $H^2(S,\Z)^{D_{12}}\cong U(2)$, up to multiplication by a constant. 
An identical argument works for $G\cong D_{10}$ and $C_2\times D_8$, 
and the B-model calculations are compatible with the previous sections. 

We turn to the elliptic curve and consider the following family of elliptic curves  in $\PP^2$ \cite{SB}  %with level $6$ structure defined 
% defined by
\begin{equation*}
 x_1 x_2 x_3 - z (x_1+x_2)(x_2+x_3)(x_3+x_1)=0. 
\end{equation*}
%by $x_1^6+x_2^3+x_3^2+z_3^{-1/6}x_1x_2x_3=0$ in %the weighted projective space $\PP^2(1,2,3)$. 
It is equipped with a translation
 $(x_1:x_2:x_3) \mapsto (1/x_2:1/x_3:1/x_1)$,
 which generates the group $H\cong C_6$, and
 the parameter $z$ is considered as a coordinate of the modular curve $X_1(6)\cong \PP^1$.
The Picard--Fuchs operator and its regular solution are given by 
%$$
%\Theta_3^2-12z_3(6\Theta_3+5)(6\Theta_3+1)
%$$
\begin{equation*}
 (8 z-1)(z+1) \Theta^2+z(16 z+7) \Theta+2z(4 z+1), \quad
 \Theta:=z \frac{d}{d z},
\end{equation*}
and 
%$$
%q_3=z_3+312z_3^2+107604z_3^3+39073568z_3^4+14645965026z_3^5+5609733423408z_3^6+\cdots.
%$$
%\begin{align*}
\begin{equation*}
 \Phi_0^{\operatorname{ell}}(z) =
 \sum_{n=0}^{\infty}
 \sum_{k=0}^{n} \begin{pmatrix}n\\k\end{pmatrix}^3z^n
 = 1+2 z+10 z^2+56 z^3+346 z^4+2252 z^5+15184 z^6+\cdots.
\end{equation*}
The parameter $z$ as a modular function on the upper half-plane $\HH$
 has the expansion
\begin{equation*}
 z(q)
 =
 \left(
  \frac{\eta(t)^3 \eta(t/6)}{\eta(t/2)^3 \eta(t/3)}
 \right)^3
 = q^{1/6}
 (1-3 q^{1/6}+3 q^{1/3}+5 q^{1/2}-18 q^{2/3}+15 q^{5/6}+\cdots)
\end{equation*}
 around $z=0$,
 where we write $q=e^{2 \pi i t}$ ($t \in \HH$) and
 $\eta(t)$ is the Dedekind eta function.
As in the K3 surface case, the fundamental period is expressed as a modular form:
% $\Phi_0^{ell}(q_3)=E_4^{1/4}(q_3)$. 
\begin{align*}
 \Phi_0^{\operatorname{ell}}(q)
 &= \frac{\eta(t/2)^6 \eta(t/3)}{\eta(t)^3 \eta(t/6)^2}
 =\frac{1}{3} \sum_{(n,m)\in \Z^2} \left( q^{(n^2+n m+m^2)/6}
 +2 q^{(n^2+n m+m^2)/3} \right) \\
 &= 1+2 q^{1/6}+4 q^{1/3}+2 q^{1/2}+2 q^{2/3}+4 q+\cdots.
\end{align*}
Therefore, we observe that the period integral and the mirror maps of the threefold $X$ are all written in terms of modular forms for $G\cong D_{12}$. 

%As in the K3 surface case, we expect that the fundamental period is expressed by modular forms.

%\begin{Prop}
%The Picard--Fuchs differential system for $S$ is generated by the operators $\mathcal{D}_i$ for $i=1,2,3$.  
%\end{Prop}
%\begin{proof}
%By the Griffith transversality, we see that $\mathcal{D}_1$ and $\mathcal{D}_2$ generate the system for $S$. 
%Indeed, the 3rd covariant derivatives of $\omega_s$ are obtained by the covariant derivatives of $\mathcal{D}_1$ and $\mathcal{D}_2$. 
%In the same manner, the $\mathcal{D}_2$ generate the system for $E$. 
%Let us now turn to $X=(S\times E)/G$ with a holomorphic 3-form $\omega_S\otimes dz$. 
%Again, b the Griffith transversality, we see that 
%$$\C\langle \Omega, \nabla_i\Omega, \nabla_i \nabla_j\Omega,\rangle_{1\le i,j \le 3}=H^{3,0}(X)\oplus H^{2,1}(X)\oplus H^{1,2}(X).$$
%This shows the existence of 3 linear relations, i.e. 3 generators of degree 2. 
%Moreover, we have 
%$$\C\langle \Omega, \nabla_i\Omega, \nabla_i \nabla_j\Omega, \nabla_i \nabla_j \nabla_k\Omega\rangle_{1\le i,j,k \le 3}=H^{3}(X,\C)$$
%This shows the existence of 9 more linear relations. 
%However, the observation for $S$ and $E$ in the beginning confirms that they are all derived from the covariant derivatives of the 3 generators of degree 2. 
%This proves the assertion. 
%\end{proof}

We may apply a similar argument for $G\cong D_{8}$. 
In this case, the period integral for the K3 surface is given by
$$
\Phi_0(z_1,z_2,z_3)=1+12(z_1+z_2+z_3)+420(z_1^2+z_2^2+z_3^2+4z_1z_2+4z_2z_3+4z_1z_3)+\cdots
$$
and the Picard--Fuchs system is generated by, for ${\bf a}=[a_1,\dots,a_6] \in \C^6$, 
\begin{align}
\mathcal{D}_{{\bf a}}=&(a_1-64a_1z_1+4(-16a_1-16a_2+3a_5)z_2-12a_4z_3)\Theta_1^2 \notag \\
&+(a_2-12a_5z_1-64a_2z_2+4(-16a_2-16a_3+3a_6)z_3)\Theta_2^2 \notag \\
&+(a_3+4(-16a_1-16a_3+3a_4)z_1-12a_6z_2-64a_3z_3)\Theta_3^2 \notag \\
&-128(a_1z_1+a_2z_2+a_3z_3)(\Theta_1\Theta_2+\Theta_2\Theta_3+\Theta_3\Theta_1) \notag \\
&- 64(a_1z_1+a_2z_2+a_3z_3)(\Theta_1+\Theta_2+\Theta_3) -12(a_1z_1+a_2z_2+a_3z_3).\notag
\end{align} 
%In the same manner as above, we determine the A-Yukawa couplings, up to multiplication by a constant, 
%$$
%\begin{bmatrix}
%        K_{t_1,t_1} & K_{t_1,t_2}  &  K_{t_1,t_3}\\
%        K_{t_2,t_1} & K_{t_2,t_2}  &  K_{t_2,t_3}\\
%        K_{t_3,t_1} & K_{t_3,t_2}  &  K_{t_3,t_3}
%        \end{bmatrix}=
%\begin{bmatrix}
%        0  &  1  &  1\\
%        1  &  0  &  1\\
%        1  &  1  &  0
%        \end{bmatrix} \cong U\oplus \langle -2\rangle,
%$$
%which is compatible with $H^2(S,\Z)^{D_{8}}\cong U(2)\oplus \langle -4\rangle$. 

% \\
%We now turn to the elliptic curve part and consider the following family of elliptic curves with level $4$ structure 
%$$
%x_1^4+x_2^4+x_3^2+z_4^{-1/4}x_1x_2x_3=0
%$$
%in the weighted projective space $\PP^2(1,1,2)$. 
%The Picard--Fuchs operator is given by 
%$$
%\Theta_4^2-4z_4(4\Theta_4+3)(4\Theta_4+1)
%$$
%and the mirror map is given by
%$$
%q_4=z_4+40z_4^2+1876z_4^3+95072z_4^4+5045474z_4^5+276107444z_4^6+\dots
%$$
%The fundamental period can be written in mirror variable $q_4$ as follows: 
%$$
%\Phi_0^{ell}(q_3)=1+12q-288q^2+\frac{12848}{3}q^3-31488q^4-\frac{6858392}{25}q^5+\frac{580562048}{45}q^6
%$$

%%%%%%%%%%%%%%%%%%%%%%%%%%%%%%%%%%%%%%%%%%%%%%%%%%%%%%%%%%%%%%%%%%%%%%%%%%%%%%%%%%%%%%%%%%%%%%%%%%%%%%%%%%%%%
%%%%%%%%%%%%%%%%%%%%%%%%%%%%%%%%%%%%%%%%%%%%%%%%%%%%%%%%%%%%%%%%%%%%%%%%%%%%%%%%%%%%%%%%%%%%%%%%%%%%%%%%%%%%%

\section{Special Lagrangian Fibrations}

Let $X$ be a Calabi--Yau manifold of dimension $n$ with a Ricci-flat K\"{a}hler metric $g$. 
Let $\kappa$ be the K\"{a}hler form associated to $g$. 
There exists a nowhere-vanishing holomorphic $n$-form $\Omega_X$ on $X$ and we normalize $\Omega_X$ by requiring 
$$
(-1)^{\frac{n(n-1)}{2}}(\frac{i}{2})^{n}\Omega_X \wedge \overline{\Omega}_X=\frac{\kappa^{n}}{n!}, 
$$
which uniquely  determines $\Omega_X$ up to a phase $e^{i \theta} \in S^1$. 
A submanifold $L\subset X$ of real dimension $n$ is called special Lagrangian if $\kappa|_{L}=0$ and $\Re(\Omega_X)|_{L}=e^{i\theta_L}\mathrm{vol}_L$ for a constant $\theta_L \in \R$. 
Here $\mathrm{vol}_L$ denotes the Riemannian volume form on $L$ induced by $g$. 
%Note that the special Lagrangian condition depends on the choice of $\Omega_X$. 
A surjective map $\pi \colon X\rightarrow B$ is called a special Lagrangian $T^n$-fibration if a generic fiber is a special Lagrangian $n$-torus $T^n$. 

%The minimal subset $D_{\pi}\subset B$ on which the restriction $\pi|_{B\setminus D_{\pi}}$ is smooth is called the discriminant locus of $\pi$.\\

In \cite{SYZ} Strominger, Yau and Zaslow proposed that mirror symmetry should be what physicists call {\it T-duality}. 
In this so-called SYZ description, a Calabi--Yau manifold $X$ admits a special Lagrangian $T^n$-fibration $\pi \colon X\rightarrow B$ near a LCSL 
and mirror Calabi--Yau manifold $Y$ is obtained by fiberwise dualization of $\pi$, modified by instanton corrections. 
The importance of this conjecture lies in the fact that it leads to an intrinsic characterization of the mirror Calabi--Yau manifold $Y$;  
it is the moduli space of these special Lagrangian fibers $T^n \subset X$ decorated with flat $\mathrm{U}(1)$-connection.

\begin{Ex} \label{SLAG Elliptic}
Let $E:=\C/(\Z+\Z\tau)$ be an elliptic curve with $\tau \in \mathbb{H}$.  
Fix a holomorphic 1-form $dz:=dx+idy$ and a K\"{a}hler form $\kappa:=dx\wedge dy$.  
A special Lagrangian submanifold $L\subset X$ is a real curve such that $\kappa|_{L}=0$ 
and $dy|_{L}=e^{i\theta_L}\mathrm{vol}_L$ for some $\theta_L \in \R$.  
The first condition is vacuous as $L$ is of real dimension 1, and  the second implies that $L$ is a {\it line}. 
%Hence $L$ is of the form $L_{c}=\{ z \in E \ | \ \mathrm{Im}(z)=c\}$ for some $0\le c < \mathrm{Im}(\tau)$.  
For example, the map 
$$
\pi_{E} \colon E=\C/(\Z+\Z\tau)\rightarrow S^1=\R/\Im(\tau), \ \ z\mapsto \Im(z)
$$ 
is a special Lagrangian smooth $T^{1}$-fibration. %with special Lagrangian sections $x\mapsto x+ic$ for any $0\le c < \mathrm{Im}(\tau)$. 
%Note that the K\"{a}hler form $\kappa$ is invariant under translations and negation $-1_E$. 
\end{Ex}

Given a Calabi--Yau manifold, finding a special Lagrangian fibration is an important and currently unsolved problem in high dimensions. 
%whereas there are some examples under relaxed conditions (see for example \cite{Gr}).    
Among others, a well-known result is Gross and Wilson's work on Borcea--Voisin threefolds \cite{GW}. 
%Recall that Borcea--Voisin threefolds are very similar to Calabi--Yau threefolds of type K; 
Recall that a Borcea--Voisin threefold is obtained from a K3 surface $S$ with an anti-symplectic (non-Enriques) involution $\iota$ and an elliptic curve $E$ 
by resolving singularities of the quotient $(S\times E)/\langle (\iota,-1_E)\rangle$ \cite{Bo,Vo}. 
%See the original papers \cite{Bo,Vo} for more details of the construction.  
%They are a class of Calabi--Yau threefolds which have been exten- sively studied.
There are two drawbacks in this case.  
Firstly, we need to allow a degenerate metric. 
Secondly, we need to work on a slice of the complex moduli space, where the threefold is realized as a blow-up of the orbifold $(S\times E)/\langle (\iota,-1_E)\rangle$. 

In this section we will construct special Lagrangian fibrations of the Calabi--Yau threefolds of type K with smooth Ricci-flat metric. 
%Indeed there are natural choices for special Lagrangian $T^3$-fibration. 
Such a fibration for the Enriques Calabi--Yau threefold was essentially constructed in \cite{GW} and our arguments below are its modifications. 
%This section is clearly influenced by \cite{GW}, and most of our arguments below are modifications of theirs. 
We state it here and do not repeat it each time in the sequel. 
%We hope the new examples provide a step forward in the SYZ program. 

%%%%%%%%%%%%%%%%%%%%%%%%%%%%%%%%%%%%%%%%%%%%%%%%%%%%%%%%%%%%%%%%%%%%%%%%%%%%%%%%%%%%%%%%%%%%%%%%%%%%%%%%%%%%%

%%%%%%%%%%%%%%%%%%%%%%%%%%%%%%%%%%%%%%%%%%%%%%%%%%%%%%%%%%%%%%%%%%%%%%%%%%%%%%%%%%%%%%%%%%%%%%%%%%%%%%%%%%%%%

\subsection{K3 Surfaces as HyperK\"{a}hler Manifolds}
Let $S$ be a K3 surface with a Ricci-flat K\"{a}hler metric $g$. 
Then the holonomy group is $\mathrm{SU}(2)\cong \mathrm{Sp}(1)$ 
and the parallel transport defines complex structures $I,J,K$ satisfying the quaternion relations: $I^{2}=J^{2}=K^{2}=IJK=-1$ 
such that $S^{2}=\{aI+bJ+cK \in \End(TS) \bigm| a^{2}+b^{2}+c^{2}=1\}$ is the possible complex structures for which $g$ is a K\"{a}hler metric.  
The period of $S$ in the complex structure $I$ is given by the normalized holomorphic 2-form $\omega_{I}(*,**):=g(J*,**)+ig(K*,**)$, 
and the compatible K\"{a}hler form is given by $\kappa_{I}(*,**):=g(I*,**)$. 
We denote, for instance, by $S_{K}$ the K3 surface $S$ with the complex structure $K$. 
\begin{center}
 \begin{tabular}{|c|c|c|}  \hline
 Complex structure & Holomorphic 2-form & K\"{a}hler form \\ \hline
 $I$  & $\omega_{I}:=\Re\omega_{I}+i\Im\omega_{I}$ & $\kappa_{I}$ \\ \hline
 $J$  & $\omega_{J}:=\kappa_{I}+i\Re\omega_{I}$ & $\kappa_{J}:=\Im\omega_{I}$ \\ \hline
 $K$  & $\omega_{K}:=\Im\omega_I+i\kappa_{I}$ & $\kappa_{K}:=\Re\omega_{I}$ \\ \hline
 \end{tabular}
 \end{center}
%Let $L\subset S_{I}$ be a special Lagrangian submanifold, i.e. $\kappa_{I}|_{L}=0$ and $\Im\omega_{I}|_{L}=0$. 
%We then have $\omega_{K}|_{L}=(\Im\omega_{I}+i\kappa_I)|_{L}=0$, which is equivalent to $TL\subset TS_{K}|_{L}$ being a complex subbundle over $L$.   
%It then follows that $L\subset S_{K}$ is a complex submanifold. 
%This argument can be reversed (Wirtinger's theorem) and we obtain the following claim.  
Recall the {\it hyperK\"{a}hler trick}, which asserts  
that a special Lagrangian $T^2$-fibration with respect to the complex structure $I$ is the same as an elliptic fibration with respect to the complex structure $K$, due to the following proposition.
\begin{Prop}[Harvey--Lawson \cite{HL}] \label{Harvey--Lawson}
A real smooth surface $L \subset S_{I}$ is a special Lagrangian submanifold  if and only if $L\subset S_{K}$ is a complex submanifold.  
 \end{Prop}

%By this proposition, a special Lagrangian $T^2$-fibration with respect to the complex structure $I$ is the same as an elliptic fibration with respect to the complex structure $K$. 
%Therefore, it is relatively easy to construct a special Lagrangian $T^2$-fibration for a certain class of K3 surfaces. 

%Proposition \ref{Harvey--Lawson} translates the study of special Lagrangian $T^2$-fibrations to that of elliptic fibrations.  
%The following useful theorem is proved by a trick of hyperK\"{a}hler rotation, that is, showing the existence of an elliptic fibration on $S_K$. 
%To the best of our knowledge

%%%%%%%%%%%%%%%%%%%%%%%%%%%%%%%%%%%%%%%%%%%%%%%%%%%%%%%%%%%%%%%%%%%%%%%%%%%%%%%%%%%%%%%%%%%%%%%%%%%%%%%%%%%%%
%%%%%%%%%%%%%%%%%%%%%%%%%%%%%%%%%%%%%%%%%%%%%%%%%%%%%%%%%%%%%%%%%%%%%%%%%%%%%%%%%%%%%%%%%%%%%%%%%%%%%%%%%%%%%

\subsection{Calabi--Yau Threefolds of Type K}
%In this section we will construct special Lagnrangian fibrations of Calabi--Yau threefolds of type K. 
%Our argument will rely on the lattice structure of $M_G$ and $N_G$ determined in Proposition \ref{G-inv lattice}.\\
Let $X$ be a Calabi--Yau threefold of type K and $S\times E\rightarrow X$ its minimal splitting covering with Galois group $G$.   
We equip $S$ with a $G$-invariant K\"{a}hler class $\kappa$, % (by averaging a K\"{a}hler form over $G$), 
which uniquely determines a $G$-invariant Ricci-flat metric on $S$ \cite{Ya}. 
We may assume that $\kappa$ is generic in the sense that $\kappa^\perp \cap H^2(S,\Z)^G=0$.  
The product Ricci-flat metric on $S\times E$ is $G$-invariant and hence descends to a Ricci-flat metric on the quotient $X$, which we will always use in the following. 
%In what follows, a metric on $X$ is the metric obtained in this way. 

\begin{Prop} \label{hyperKahler rotation}
There exists $\omega_I \in \mathcal{D}_S^G$ such that its hyperK\"{a}hler rotation $S_K$ admits an elliptic fibration $\pi_{S} \colon S_K\rightarrow \mathbb{P}^1$ with a $G$-stable multiple-section.
In other words, $S_I$ admits a special Lagrangian $T^2$-fibration $\pi_{S} \colon S_I\rightarrow \mathbb{P}^1$ with a special Lagrangian multiple-section which is $G$-stable as a set. 
%Moreover, we can construct the fibration so that it has no reducible fibers. 
\end{Prop}
\begin{proof}
By Theorem \ref{G-inv lattice}, the transcendental lattice $N_G=H^2(S,\Z)_{C_2}^{H}$ always contains $U(k)$ for some $k \in \N$. 
We denote by $e,f$ the standard basis of $U(k)$. 
We choose a generic $\omega_I \in \mathcal{D}_S^G$ with $\Im \omega_I \in U(k)^\perp_{N_G}\otimes \R$. 
Then it follows that $e,f$ are contained in $NS(S_K)$.  
By \cite[Lemma 1.7]{OS}, we may assume that $f$ is nef after applying a sequence of reflections with respect to $(-2)$-curves in $NS(S_K)$, if necessary.
Hence the linear system of $f$ induces an elliptic fibration $\pi_{S} \colon S_K\rightarrow \mathbb{P}^1$.
Let $C$ be an irreducible curve $C$ in $S_K$ such that $\langle C,f \rangle>0$.
Then $C$ is a multiple-section of the elliptic fibration with multiplicity $\langle C,f \rangle$.
Thus $\cup_{g\in G} \, g(C)$ is a $G$-stable multiple-section.
%Moreover, if we choose a sufficiently generic period $\omega_I$ as above, 
%we see that $NS(S_K)=\omega_K ^\perp$ is generated by $e$ and $f$, and thus the map $\pi_S$ has no reducible fibers.     
\end{proof}

It is possible to show the existence of a special Lagrangian fibration for a {\it generic} choice of complex structure of $S$ with a {\it numerical} multiple-section (see \cite[Proposition 1.3]{GW}). 
There is in general no canonical choice of such a fibration in contrast to the Borcea--Voisin case.  
Henceforth we denote simply by $S$ the K3 surface $S_I$ with the complex structure $I$ obtained in Proposition \ref{hyperKahler rotation}.  

%\begin{Prop} \label{ReducibleFiber}
%Assume that the period $\omega_I$ is chosen generically in Proposition \ref{hyperKahler rotation}, 
%then the fibration $\pi_S$ has no reducible fiber. 
%\end{Prop}
%\begin{proof}
%%We first observe that $(f)^\perp_{NS(S_K)}/f= H^2(S,\Z)_H$. 
%We have $\langle \kappa,x\rangle=0$ for any $(-2)$-curve class $x \in (f)^\perp_{NS(S_K)}/f=H^2(S,\Z)_H$ because $\kappa \in H^2(S,\Z)^G$. 
%Since the lattice $(f)^\perp_{NS(S_K)}$ contains no root, the fibration $\pi_S$ has no reducible fiber. 
%\end{proof}

\begin{Prop} \label{equiv}
The action of $H$ on $S_{K}$ is holomorphic and that of $\iota$ is anti-holomorphic, where $K$ is the complex structure obtained in Proposition \ref{hyperKahler rotation}.   
Moreover the action of $G$ on $S$ descends to the base $\mathbb{P}^1$ in such a way that 
the fibration $\pi_{S} \colon S\rightarrow \mathbb{P}^1$  is $G$-equivariant. 
\end{Prop}
\begin{proof}
%(c.f.\cite{GW}[Lemma 2.2]). 
Recall that the K\"{a}hler form $\kappa$ is $G$-invariant. 
The metric $g$ is still Ricci-flat after a hyperK\"{a}hler rotation and hence the $G$-action is isometry with respect to the metric $g$ on $S_{K}$. 
The holomorphic 2-form $\omega_{K}$ is harmonic and so is $g^{*}\omega_{K}$ for any $g\in G$.  
For any $h\in H$, we hence have $h^{*}\omega_{K}=\omega_{K}$ as a 2-form. 
It then follows that the $H$-action is holomorphic on $S_{K}$. 
In the same manner, we can show that $\iota^*\omega_K=-\overline{\omega_K}$ as a 2-form and thus $\iota$ is anti-holomorphic. 
Since $H$ leaves the fiber
% and multiple-section
 class of $\pi_S$ invariant and $\iota$ simply changes its sign, %\iota(F) is a divisor and \langle \iota(F),E \rangle=0, hence vibration is preserved
the action of $G$ on $S_{K}$ descends to the base $\mathbb{P}^1$ in such a way that 
the fibration $\pi_S \colon S_{K}\rightarrow \mathbb{P}^1$ is $G$-equivariant. 
\end{proof}

Let $\pi_S \colon S\rightarrow \mathbb{P}^1 \cong S^2$ be the special Lagrangian fibration in Proposition \ref{hyperKahler rotation}.  
Combining the maps $\pi_S$ and $\pi_E$, we obtain a  special Lagrangian $T^{3}$-fibration $\pi_{S\times E} \colon S\times E \rightarrow S^{2}\times S^{1}$ with respect to the Ricci-flat metric on $S\times E$. 
Moreover, by Proposition \ref{equiv}, $\pi_{S\times E}$ induces a map 
$$
\pi_{X} \colon X=(S\times E)/G\rightarrow B:=(S^2\times S^1)/G 
$$
with a multiple-section. 
%Euler number $\chi(X)$ is concentrated in the singular fibers of the $T^3$-fibration, and so we expect to see locally how the Euler number change sign when we pass to the mirror Calabi--Yau threefold. 
%The Euler number $\chi(F)$ of singular fibbers $F$ add up to the Euler number $\chi(X)$. 

\begin{Prop}\label{SLAG}
The map $\pi_{X} \colon X\rightarrow B$ is a special Lagrangian $T^{3}$-fibration such that each (possibly singular) fiber has Euler number zero. 
\end{Prop}
\begin{proof}
%The assertion follows from a simple calculation. 
Let $p \colon S^{2}\times S^{1} \rightarrow B$ denote the quotient map.
We define $S^{[H]}$ to be the set of points in $S$ fixed by some $h \in H$.
Similarly we define $(S^2)^{[G\setminus H]}$.
For any generic $b \in B$ in the sense that the fiber $\pi_{S\times E}^{-1}(p^{-1}(b))$ is smooth and
$$
b \not\in p\bigl( (\pi_S(S^{[H]}) \cup (S^2)^{[G\setminus H]}) \times S^1 \bigr),
%)\Big(\big((D_{\pi_{S}} \cup )\times S^1 \big)\cup \big(S^2 \times (S^1)^\iota\big) \Big)=\emptyset, 
$$
the fiber $\pi_{S\times E}^{-1}(p^{-1}(b))$ consists of some copies of $T^3$ and the fiber $\pi_X^{-1}(b)=\pi_{S\times E}^{-1}(p^{-1}(b))/G$ is again $T^3$. 
This shows that $\pi_{X}$ is a special Lagrangian $T^{3}$-fibration because the special Lagrangian condition is clearly preserved by the $G$-action.  
Moreover, for any $c=(c_1,c_2)\in S^2 \times S^1$, we have the Euler number $\chi_{\operatorname{top}}(\pi_{S\times E}^{-1}(c))=\chi_{\operatorname{top}}(\pi_{S}^{-1}(c_1)\times S^1)=0$. 
It then follows that $\chi_{\operatorname{top}}(\pi_X^{-1}(b))=\chi_{\operatorname{top}}\bigl(\pi^{-1}_{S\times E}(p^{-1}(b))\bigr)/|G|=0$ for any $b \in B$. 
%$$
%\chi(\pi_X^{-1}(b))=\frac{\chi(\pi^{-1}_{S\times E}(p^{-1}(b)))}{|G|}=0.
%$$ 
\end{proof}

\begin{Prop} \label{SLAG2}
The base space $B$ is topologically identified either as the $3$-sphere $S^{3}$ or an $S^1$-bundle over $\mathbb{RP}^2$. 
\end{Prop}
\begin{proof}
In the same manner as above, it is easy to see that the map $(S\times E)/H\rightarrow (S^2\times S^1)/H$ is a special Lagrangian $T^{3}$-fibration. 
Since the $H$-actions on $S^2$ and $S^1$ are rotations, the base $(S^2\times S^1)/H$ is topologically isomorphic to $S^2 \times S^1$. 
Let $\overline{\iota}$ be the generator of $G/H$.
The action of $\overline{\iota}$ on the first factor $S^{2}\cong \mathbb{P}^1$ is anti-holomorphic and hence is identified with either 
$$
z\mapsto \overline{z} \ \text{with} \ (S^{2})^{\overline{\iota}}=\R\cup \infty, \ \ \ \text{or} \ \ \  z\mapsto -\frac{1}{\overline{z}} \ \text{with} \ (S^{2})^{\overline{\iota}}=\emptyset. 
$$ 
The action of ${\overline{\iota}}$ on the second factor $S^1=\{|z|=1\} \subset \C$ is identified with the reflection along the real axis with $(S^{1})^{\overline{\iota}}=S^0$. 
The case when $(S^{2})^{\overline{\iota}}=\R\cup \infty$ was previously studied in \cite[Section 3]{GW} and the base $B=(S^2\times S^1)/\langle {\overline{\iota}} \rangle$ is topologically identified with $S^3$. 
On the other hand, when $(S^{2})^{\overline{\iota}}=\emptyset$, the base $B=(S^2\times S^1)/\langle {\overline{\iota}} \rangle$ is endowed with an $S^1$-bundle structure via the first projection:
$$
B=(S^2\times S^1)/\langle {\overline{\iota}} \rangle \rightarrow S^2/\langle {\overline{\iota}} \rangle \cong \mathbb{RP}^2. 
$$ 
\end{proof}
Note that $B$ need not be a $3$-sphere $S^{3}$ as $X$ is not simply-connected, but we do not know whether or not the latter case really occurs. 

For discussion on Brauer groups, dual fibrations and their sections, we refer the reader to \cite[comment after Remark 3.12]{Gr2}.  
%There are some other topologically distinct Calabi--Yau threefolds which have SYZ fibrations without section and whose Jacobian is the original. 
Another interesting problem is to relate the Hodge theoretic mirror symmetry with the SYZ description.
Aspects of this have been studied for Borcea--Voisin threefolds in \cite{GW} (later more generally in \cite{Gr2}), 
where they recover the mirror map from the Leray spectral sequence associated to a special Lagrangian $T^3$-fibration. 
For a Borcea--Voisin threefold, this is the first description of the mirror map and conjecturally their method can be applied to a larger class of Calabi--Yau manifolds. 
The difficulty to tackle the problem in the present case is that the singular fibers may be reducible and the spectral sequence is more involved.

\par\noindent{\scshape \small
Graduate School of Mathematical Sciences, The University of Tokyo,\\
3-8-1 Komaba, Maguro-ku, Tokyo, 153-8914, Japan}
\par\noindent{\ttfamily kenji.hashimoto.math@gmail.com}
\break
\par\noindent{\scshape \small
%Center of Mathematical Sciences and Applications\\ 
Department of Mathematics, Harvard University\\
One Oxford Street, Cambridge MA 02138 USA }
\par\noindent{\ttfamily kanazawa@cmsa.fas.harvard.edu}

\end{document}